\documentclass[12pt]{article} 

\usepackage[utf8]{inputenc}

\usepackage[T1]{fontenc} 

\usepackage{geometry} 
\usepackage{graphicx} 

\usepackage{titlesec}

\usepackage[english]{babel}
\usepackage{amsmath}
\usepackage{booktabs} 
\usepackage{array} 
\usepackage{paralist} 
\usepackage{verbatim} 
\usepackage{subfig} 
\usepackage{ucs}
\usepackage{dsfont}
\usepackage{amsfonts}
\usepackage{mathabx}
\usepackage{amsthm}
\usepackage{thmtools}
\usepackage{tikz}
\usepackage[colorlinks=true,pdfstartview=FitH,bookmarks=false]{hyperref}
\usepackage[all]{xy}

\usepackage{fancyhdr} 
\usepackage{sectsty}
\allsectionsfont{\sffamily\mdseries\upshape} 

\usepackage[nottoc,notlof,notlot]{tocbibind} 
\usepackage[titles,subfigure]{tocloft} 

\theoremstyle{definition}

\usepackage{authblk}

\theoremstyle{plain}
\newtheorem{coro}{Corollary}[section]
\newtheorem{lem}{Lemma}[section]
\newtheorem{theorem}{Theorem}[section]
\newtheorem{conj}{Conjecture}[section]
\newtheorem{prop}{Proposition}[section]

\newtheorem{question}{Question}

\theoremstyle{remark}
\newtheorem{rem}{Remark}[section]

\newcommand\N{\mathbb{N}}
\newcommand\Z{\mathbb{Z}}
\newcommand\Q{\mathbb{Q}}
\newcommand\R{\mathbb{R}}

\newcommand\D{\mathbb{D}}
\newcommand\Ss{\mathds{S}}
\newcommand\Tt{\mathds{T}}

\newcommand\ra{\rightarrow}

\newcommand\e{\mathrm{e}}

\newcommand\C{\mathcal{C}}

\newcommand{\ti}[1]{\widetilde{#1}}

\newcommand{\Diff}{\mathrm{Diff}}
\newcommand{\Ham}{\mathrm{Ham}}
\newcommand{\Ang}{\mathrm{Ang}}
\newcommand{\Cal}{\mathrm{Cal}}

\newcommand{\tin}{t\in[0,1]}
\newcommand{\tphi}{\widetilde{\phi}}

\newcommand{\id}{\mathrm{id}}
\newcommand{\Leb}{\mathrm{Leb}}
\newcommand{\Symp}{\mathrm{Symp}}
\newcommand{\Homeo}{\mathrm{Homeo}}
\date{February 2021}

\title{The Calabi invariant for Hamiltonian diffeomorphisms of the unit disk}
\author{Benoît Joly}

\begin{document}

\maketitle

\abstract{In this article, we study the Calabi invariant on the unit disk usually defined on compactly supported Hamiltonian diffeomorphisms of the open disk. In particular we extend the Calabi invariant to the group of $C^1$ diffeomorphisms of the closed disk which preserves the standard symplectic form. We also compute the Calabi invariant for some diffeomorphisms of the disk which satisfies some rigidity hypothesis.}
\tableofcontents

			\section{Introduction}
Let us begin with some basic definitions of symplectic geometry.\\
 		
Let us consider $(M^{2n},\omega)$ a \emph{symplectic manifold}, meaning that $M$ is an even dimensional manifold equipped with a closed non-degenerate differential 2-form $\omega$ called the \emph{symplectic form}. We suppose that $\pi_2(M)=0$ and that $\omega$ is exact, meaning that there exists a $1$-form $\lambda$, called a \emph{Liouville form}, which satisfies $d\lambda=\omega$.\\

Let us consider a time-dependent vector field $(X_t)_{t\in \R}$ defined by the equation 
\begin{equation}\label{Xt}
dH_t=\omega(X_t, .),
\end{equation}
where
\begin{align*}
H : \ &\R\times M\ra \R\\
& (t,x) \mapsto H_t(x)
\end{align*}
is a smooth function  $1$-periodic on $t$, meaning that $H_{t+1}=H_t$ for every $t\in\R$. The function $H$ is called a \textit{Hamiltonian function}. If the vector field $(X_t)_{t\in \R}$ is complete, it induces a family $(f_t)_{t\in\R}$ of diffeomorphisms of $M$ that preserve s$\omega$, also called \emph{symplectomorphisms} or \emph{symplectic diffeomorphisms}, satisfying the equation
$$\frac{\partial}{\partial t}f_t(z)=X_t(f_t(z)).$$
In particular the family $I=(f_t)_{\tin}$ defines an isotopy from $\id$ to $f_1$. The map $f_1$ is called a \textit{Hamiltonian diffeomorphism}. It is well known that the set of Hamiltonian diffeomorphisms of a symplectic manifold $M$ is a group which we denote $\Ham(M,\omega)$, we refer to \cite{McD} for more details.\\

Let us consider $(M,\omega)$ a symplectic manifold which is boundaryless, $\pi_2(M)=0$ and such that $\omega$ is exact. We say that $H$ is a \emph{compactly supported Hamiltonian function} if there exists a compact set $K\subset M$ such that $H_t$ vanishes outside $K$ for every $t\in\R$. A compactly supported Hamiltonian function induces a \emph{compactly supported Hamiltonian diffeomorphism} $f$. Such a map is equal to the identity outside a compact subset of $M$. Let us consider a compactly supported Hamiltonian diffeomorphism $f$ and $\lambda$ a Liouville form on $M$. The form $f^*\lambda-\lambda$ is closed because $f$ is symplectic but we have more, it is exact. More precisely there exists a unique compactly supported function $A_f:M\ra \R$, also called \textit{action function}, such that 
$$dA_{f}=f^*\lambda-\lambda.$$
In the literature the \textit{Calabi invariant} $\Cal(f)$ of $f$ is defined as the mean of the function $A_f$ and we have
\begin{equation}\label{CalIntro}
\Cal(f)=\int_M A_f\omega^n,
\end{equation}
where $\omega^n=\omega\wedge...\wedge \omega$ is the volume form induced by $\omega$, see \cite{McD} for more details. We will prove later that the number  $\Cal(f)$ does not depend on the choice of $\lambda$.\\

Let us give another equivalent definition of the Calabi invariant for a compactly supported Hamiltonian diffeomorphism $f$. We note $H$ a compactly supported Hamiltonian function defining $f$. The Calabi invariant of $f$ can also be defined by the equation
\begin{equation}\label{Calintro}
\Cal(f)=(n+1)\int_0^1\int_MH_t\omega^n dt.
\end{equation}
To prove that $\int_M A_f\omega^n$ does not depend on the choice of the Liouville form $\lambda$, one may use the fact that the action function $A_f$ satisfies
\begin{equation}\label{actionham}
A_f(z)=\int_0^1(\iota(X_s)\lambda+H_s)\circ f_s(z)ds,
\end{equation}
where $(X_s)_{s\in\R}$ is the time dependent vector field induced by $H$ by equation (\ref{Xt}) and $(f_s)_{s\in \R}$ is the isotopy induced by the vector field $(X_s)_{s\in\R}$. Moreover, $\int_0^1\int_MH_t\omega^n dt$ does not depend on the compactly supported Hamiltonian function $H$ defining $f$.\\

The function $\Cal$ defines a real valued morphism on the group of compactly supported Hamiltonian diffeomorphisms of $M$ and thus it is a conjugacy invariant. It is an important tool in the study of difficult problems such as the description of the algebraic structure of the groups $\Ham(M,\omega)$: A.Banyaga proved in \cite{BAN} that the kernel of the Calabi invariant is always simple, which means that it does not contain nontrivial normal subgroups. \\

In this article, we study the case of the dimension two and more precisely the case of the closed unit disk which is a surface with boundary. We denote by $||.||$ the usual Euclidian norm on $\R^2$, by $\D$ the closed unit disk and by $\Ss^1$ its boundary. The group of  $C^1$ orientation preserving diffeomorphisms of $\D$ will be denoted by $\Diff_+^1(\D)$. We consider $\Diff^1_\omega(\D)$ the group of $C^1$ symplectomorphisms of $\D$ which preserve the normalized standard symplectic form $\omega=\frac{1}{\pi}du\wedge dv,$ written in cartesian coordinates $(u,v)$. In the case of the disk, the group $\Diff^1_\omega(\D)$ is contractile, see \cite{HIR} for a proof, and coincides with the group of Hamiltonian diffeomorphisms of $\D$. Moreover, the $2$-form $\omega$ induces the Lebesgue probability measure denoted by $\Leb$ and the symplectic diffeomorphisms  are the $C^1$ diffeomorphisms of $\D$ which preserve the Lebesgue measure and the orientation.\\

Let us begin by the case of the unit open disk $\mathring{\D}$. The open disk is boundaryless hence we already have two equivalent definitions of the Calabi invariant given by equations \ref{CalIntro} and \ref{Calintro} on the set of compactly supported symplectic diffeomorphisms of $\mathring{\D}$. Let us give a third one. A. Fathi in his thesis \cite{FAT} gave a dynamical definition which is also described by J.-M. Gambaudo and \'E. Ghys in \cite{GHYS}: if we consider an isotopy $I=(f_t)_{\tin}$ from $\id$ to $f$, there exists an angle function $\Ang_I:\mathring{\D}\times\mathring{\D}\backslash \Delta\ra \R$ where $\Delta$ is the diagonal of $\mathring{\D}\times\mathring{\D}$ such that for each $(x,y)\in\mathring{\D}\times\mathring{\D}\backslash \Delta$, the quantity $2\pi\Ang_I(x,y)$ is the variation of angle of the vector $f_t(y)-f_t(x)$ between $t=0$ and $t=1$. If $f$ is a compactly supported $C^1$ symplectic diffeomorphism then this angle function is integrable (see section \ref{threeExtensions}) and it holds that
\begin{equation}\label{Calangle}
\Cal(f)=\int_{\mathring{\D}\times\mathring{\D}\backslash \Delta} \Ang_I(x,y)d\Leb(x)d\Leb(y),
\end{equation}
where the integral does not depend on the choice of the isotopy.\\

In this article we will give an answer to the following question.
\begin{question}
How to define an extension of the Calabi invariant to the group $\Diff^1_\omega(\D)$?
\end{question}

M. Hutchings \cite{HUT} extended the definition given by equation \ref{Calintro} to the $C^1$ symplectic diffeomorphisms which are equal to a rotation near the boundary. In another point of view, V. Humilière \cite{HUM3} extended the definition given by equation \ref{Calintro} to certain group of compactly supported symplectic homeomorphisms of an exact symplectic manifold $(M,\omega)$ where a compactly supported symplectic homeomorphism $f$ of $M$ is a $C^0$ limit of a sequence of Hamiltonian diffeomorphisms of $M$ supported on a common compact subset of $M$.\\

In the case of the open disk, for a compactly supported symplectomorphism $f$, the choice of the isotopy class of $f$ is natural. But if $f$ is a symplectic diffeomorphism of the closed disk such that its restriction to the open disk is not compactly supported then there is no such natural choice of an isotopy from $\id$ to $f$. \\

The rotation number is a well-known dynamical tool introduced by Poincaré in \cite{POIN} on the group $\Homeo_+(\Ss^1)$ of homeomorphisms of $\Ss^1$ which preserve the orientation. Let us consider the set of homeomorphisms $\ti{g}:\R\ra\R$ such that $\ti{g}(x+1)=\ti{g}(x)$, denoted $\ti{\Homeo}_+(\Ss^1)$. One may prove that there exists a unique $\ti{\rho}\in\R$ such that for each $z\in\R$ and $n\in\Z$ we have $|\ti g^n(z)-z-n\ti\rho|<1$. The number $\ti{\rho}=\ti\rho(\ti g)$ is called the \emph{rotation number} of $\ti g$. Let us consider $g\in\Homeo_+(\Ss^1)$ and two lifts $\ti g$ and $\ti g'$ of $g$ in $\ti\Homeo_+(\Ss^1)$, there exists $k\in\Z$ such that $\ti g=\ti g'+k$ and so $\ti\rho(\ti g)=\ti\rho(\ti g')+k$. Consequently we can define a map $\rho:\Homeo_+(\Ss^1)\ra \Tt^1$ such that $\rho(g)=\ti\rho(\ti g)+ \Z$ where $\ti g$ is a lift of $g$. The number $\rho(g)$ is called the \emph{rotation number} of $g$. We give further details about the rotation number in the next section.\\

We now state the results of this article. The following proposition allows us to consider a natural choice of an action function of a symplectomorphism of the closed disk.

\begin{prop}
Let us consider $f\in\Diff^1_\omega(\D)$, $A_f:\D\ra\R$ a $C^1$ function such that $dA_f=f^*\lambda-\lambda$ and $\mu$ an $f$ invariant Borel probability measure supported on $\Ss^1$. Then the number $\int_{\Ss^1}A_f d\mu$ does not depend on the choice of $\mu$ and $\lambda$.
\end{prop}

The first theorem follows.

\begin{theorem}\label{thm1}
For each $f\in \Diff^1_\omega(\D)$ there exists a unique function $A_f:\D\ra\R$ such that $dA_f=f^*\lambda-\lambda$ and $\int_{\Ss^1} A_fd\mu=0$ where $\lambda$ is a Liouville form and $\mu$ a $f$-invariant probability measure on $\Ss^1$. The map $\Cal_1:\Diff^1_\omega(\D) \ra \R $ defined by
$$\Cal_1(f)=\int_\D A_f(z) \omega(z)$$
does not depend on the choice of $\lambda$ and $\mu$. Moreover the map $\Cal_1$ is a homogeneous quasi-morphism that extends the Calabi invariant.
\end{theorem}

In another direction, the definition given by equation \ref{Calintro} and the definition given by equation \ref{Calangle} are based on isotopies. Then we consider the universal cover $\ti{\Diff}^1_\omega(\D)$ of $\Diff^1_\omega(\D)$ which is composed of couples $\ti f=(f,[I])$ where $f\in\Diff^1_\omega(\D)$ and $[I]$ is an homotopy class of isotopies from $\id$ to $f$. We will prove that for $f\in\Diff^1_\omega(\D)$ and $I$ an isotopy from $\id$ to $f$, the angle function $\Ang_I$ does not depend on the choice of $I\in[I]$. Hence, for $\ti f=(f,[I])\in\ti{\Diff}^1_\omega(\D)$ we can denote $\Ang_{\ti f}=\Ang_I$ for $I\in [I]$.\\ 
Moreover, for a diffeomorphism $f\in\Diff^1(\D)$ two isotopies $I=(f_t)_{\tin}$ and $I'=(f'_t)_{\tin}$ from $\id$ to $f$ are homotopic if and only if there restriction $I|_{\Ss^1}$ and $I'|_{\Ss^1}$ to $\Ss^1$ are homotopics and so define the same lift $\ti {f|_{\Ss^1}}$ of $f|_{\Ss^1}$ on the universal cover over $\Ss^1$. Hence it is equivalent to consider $\ti{\Diff}^1_\omega(\D)$ as the set of couples $\ti f=(f,\ti\phi)$ where $f\in \Diff^1_\omega(\D)$ and $\tphi$ a lift of $f|_{\Ss^1}$ to the universal cover of $\Ss^1$.


\begin{theorem}\label{thm2}
Let us consider an element $\ti f$ of $\ti{\Diff}^1_\omega(\D)$. The number 
$$\ti{\Cal}_2(\ti f)=\int_{\D^2\backslash\Delta} \Ang_{\ti f}(x,y)\omega(x)\omega(y),$$
defines a morphism $\ti{\Cal}_2:\ti{\Diff}^1_\omega(\D)\ra\R$ which induces a morphism $\Cal_2:\Diff^1_\omega(\D)\ra \Tt^1$ defined for every $f\in\Diff^1_\omega(\D)$ by
$$\Cal_2(f)=\ti{\Cal}_2(\ti f)+\Z,$$
where $\ti f$ is a lift of $f$ in $\ti{\Diff}^1_\omega(\D)$.
\end{theorem}

Along the same lines, we have the following result.\\

\begin{theorem}\label{thm3}
Let us consider an element $(f,\tphi)$ of $\ti{\Diff}^1_\omega(\D)$. There exists a Hamiltonian function $(H_t)_{\tin}$ such that $H_t$ is equal to $0$ on $\Ss^1$ for every $t\in\R$ which induces an isotopy $(\phi_t)_{\tin}$ from $\id$ to $f$ where the lifted isotopy $(\ti\phi_t)_{\tin}$ satisfies $\ti\phi_1=\ti\phi$ . The number
$$\ti\Cal_3(f,\tphi)=\int_0^1\int_\D H_t(z)\omega(z)dt,$$
does not depend on the choice of the Hamiltonian function $H$. Moreover the map $\ti\Cal_3:\ti\Diff^1_{\omega(\D)}\ra\R$ is a morphism and induces a morphism $\Cal_3: \Diff^1_{\omega(\D)} \ra\Tt^1$ defined by
$$\Cal_3(f)=\ti{\Cal}_3(f,\tphi)+\Z.$$
\end{theorem}

\begin{rem}
We have the following commutative diagram
$$\xymatrix{
    \ti{\Diff}^1_\omega(\D) \ar[r]^{\ti{\pi}} \ar[d]_{\ti{\Cal}_i}  & \Diff^1_\omega(\D) \ar[d]^{\Cal_i}  \\
    \R \ar[r]^{\pi}  & \Tt^1 \\
  }$$ 
where $i\in\{2,3\}$.
\end{rem}
The link between these three extensions is given by the following result:

\begin{theorem}\label{thmlink}
The morphisms $\ti{\Cal}_2$ and $\ti{\Cal}_3$ are equal and for $\ti f=(f,\ti{\phi})\in \ti{\Diff}^1_\omega(\D)$ we have the following equality
$$\ti{\Cal}_2(\ti f)=\Cal_1(f)+\ti{\rho}(\ti{\phi}).$$
Moreover the maps $\Cal_1$, $\ti\Cal_2$, $\Cal_2$, $\ti\Cal_3$ and $\Cal_3$ are continuous in the $C^1$ topology. 
\end{theorem}

In the following, $\ti\Cal_2$ and $\ti\Cal_3$ will be denoted $\ti\Cal$. Since the morphism $\ti{\Cal}$ and the quasi-morphism $\Cal_1$ are not trivial we obtain the following corollary about the perfectness of the groups $\ti{\Diff}^1_\omega(\D)$ and $\Diff^1_\omega(\D)$. Recall that a group $G$ is said to be perfect if it is equal to its commutator subgroup $[G,G]$ which is generated by the commutators $[f,g]=f^{-1}g^{-1}fg$ where $f$ and $g$ are elements of $G$

\begin{coro}
The groups $\ti{\Diff}^1_\omega(\D)$ and $\Diff^1_\omega(\D)$ are not perfect.
\end{coro}

The non simplicity of those groups were already known since the group of compactly supported Hamiltonian diffeomorphisms is a non trivial normal subgroup of $\Diff^1_\omega(\D)$. The questions of the simplicity and the perfectness of groups of diffeomorphisms and Hamiltonian diffeomorphism have a long story, especially the case of the group of area-preserving and compactly supported homeomorphisms of the disk $\D$. The question appears on McDuff and Salamon's list of open problems in \cite{McD} and we can refer for example to \cite{BAN,Bou08,EPP12,Fat80,LR10a,LR10b,Oh10,OM07}. Recently D. Cristofaro-Gardiner, V. Humilière, S. Seyfaddini in \cite{HUM2} proved that the connected component of $\id$ in the group of area-preserving homeomorphisms of the unit disk $\D$ is not simple. The proof requires the study of the Calabi invariant on the group of compactly supported Hamiltonian of $\D$ but also strong arguments of symplectic geometry as Embedded Contact Homology (also called ECH) developed by M. Hutchings and D. Cristofaro-Gardiner in \cite{HUM2}.\\ 

To give an illustration of the extension we compute the Calabi invariant $\Cal_1$ of non trivial symplectomorphisms in sections \ref{computation} and \ref{examples}. We study the Calabi invariant $\Cal_1$ of some \emph{irrational pseudo rotations}. An irrational pseudo-rotation of the disk is an area-preserving homeomorphism $f$ of $\D$ that fixes $0$ and that does not possess any other periodic point. To such a homeomorphism is associated an irrational number $\overline{\alpha}\notin \Q/\Z$, called the \textit{rotation number} of $f$ that measures the rotation number of every orbit around $0$ and consequently is equal to the rotation number of the restriction of $f$ on $\Ss^1$. We refer to the next section for more details.

The following results of this paper are well-inspired by M. Hutchings's recent work. M. Hutching proved as a corollary in \cite{HUT} that the Calabi invariant $\Cal_3$ of every $C^\infty$ irrational pseudo rotation $f$ of the closed unit disk $\D$ such that $f$ is equal to a rotation near the boundary is equal to the rotation number of $f$. This means that for an irrational pseudo rotation $f$ which is equal to a rotation near the boundary, $\Cal_1(f)$ is equal to $0$. The proof uses strong arguments of symplectic geometry such as the notion of open-books introduced by Giroux (see \cite{GIR} for example) and the Embedded Contact Homology theory. We want to adopt a more dynamical point of view and we partially answer the following question.

\begin{question}\label{question}
Is the Calabi invariant $\Cal_1(f)$ equal to $0$ for every $C^1$ irrational pseudo rotation $f$ of $\D$?
\end{question}

With the continuity of $\ti\Cal$ in the $C^1$ topology, we can deduce the first result of $C^1$-rigidity as the following result.

\begin{theorem}
Let  $f$ be a $C^1$ irrational pseudo rotation of $\D$. If there exists a sequence $(g_n)_{n\in\N}$ in $\Diff_\omega^1(\D)$ of $C^1$ diffeomorphisms of finite order which converges to $f$ for the $C^1$ topology, then 
$$\Cal_1(f)=0.$$
\end{theorem}

\begin{coro}\label{Coro2}
Let $f$ be a $C^1$ irrational pseudo rotation of $\D$. If there exists a sequence $(n_k)_{k\in\N}$ such that $f^{n_k}$ converges to the identity in the $C^1$ topology, then we have
$$\Cal_1(f)=0.$$
\end{coro}

The morphisms $\ti{\Cal}$ and $\Cal$ are not continuous in the $C^0$ topology, see proposition \ref{notcontinuous}. Nevertheless, by a more precise study of the definition of $\Cal$ we obtain a $C^0$-rigidity result as follows.

\begin{theorem}\label{Myresult}
Let $f$ be a $C^1$ irrational pseudo rotation of $\D$. If there exists a sequence $(n_k)_{k\in\N}$ of integers such that $(f^{n_k})_{k\in\N}$ converges to the identity in the $C^0$ topology, then we have
$$\Cal_1(f)=0.$$
\end{theorem}

There are already general results of $C^0$-rigidity of the pseudo-rotations. Bramham proved \cite{Bram} that every $C^\infty$ irrational pseudo-rotation $f$ is the limit, for the $C^0$ topology, of a sequence of periodic $C^\infty$ diffeomorphisms. Bramham \cite{Bram2} also proved that if we consider an irrational pseudo-rotation $f$ whose rotation number is super Liouville (we will define what it means later) then $f$ is $C^0$-rigid. That is, there exists a sequence of iterates $f^{n_j}$ that converges to the identity in the $C^0$-topology as $n_j\ra \infty$. Le Calvez \cite{CAL1} proved similar results for $C^1$ irrational pseudo-rotation $f$ whose restriction to $\Ss^1$ is $C^1$ conjugate to a rotation.\\

Then for $f$ a $C^1$ pseudo-rotation of the disk $\D$ the results of Bramham and Le Calvez provide a sequence of periodic diffeomorphisms $(g_n)_{n\in\N}$ which converges to $f$, the diffeomorphism $g_n$ may not be area-preserving but let us hope to completely answer question \ref{question}.\\

In a last section we give some examples where the rotation number of a pseudo-rotation satisfy some algebraic properties and where the hypothesis of Theorem \ref{Myresult} and Corollary \ref{Coro2} are satisfied.\\

\textbf{Organization}

We begin to give some additional preliminaries in section $1$. In a second section we give the formal definitions of the Calabi invariant of equations \ref{CalIntro}, \ref{Calintro} and \ref{Calangle} and their natural extensions given by Theorems \ref{thm1}, \ref{thm2} and \ref{thm3}. In section $3$ we give the proof the link betweens these extensions given by Theorem \ref{thmlink}. The last section concerns the results about the computation of the Calabi invariant for pseudo rotations.\\

				\section{Preliminaries}\label{preliminaries}
			
\textbf{Invariant measures.} Let us consider $f$ a homeomorphism of a topological space $X$. A Borel probability measure $\mu$ is $f$-invariant if for each Borel set $A$ we have 
$$\mu(f^{-1}(A))=\mu(A).$$
In other terms, the push forward measure $f_*\mu$ is equal to $\mu$. We denote by $\mathcal{M}(f)$ the set of $f$-invariant probability measures on $X$. It is well-known that the set $\mathcal{M}(f)$ is not empty if $X$ is compact.\\
			
			For a probability measure $\mu$ on $\D$ we will note $\Diff^1_\mu(\D)$ the subgroup of $\Diff^1_+(\D)$ that is the set of orientation preserving $C^1$ diffeomorphisms which preserve $\mu$. \\

\textbf{Quasi-morphism.} A function $F:G\ra\R$ defined on a group $G$ is a \emph{homogeneous quasi-morphism} if
\begin{enumerate}
\item there exists a constant $C\geq0$ such that for each couple $f,g$ in $G$ we have $|F(f\circ g)-F(f)-F(g)|<C$,	
\item for each $n\in \Z$ we have $F(f^n)=nF(f)$.
\end{enumerate}

\textbf{Rotation numbers of homeomorphisms of the circle.} The rotation number is defined on the group $\Homeo_+(\Ss^1)$ of homeomorphisms of $\Ss^1$ which preserve the orientation. We begin to give the definition of the rotation number on the lifted group $\ti{\Homeo}_+(\Ss^1)$ which is the set of homeomorphisms $\ti{g}:\R\ra\R$ such that $\ti{g}(x+1)=\ti{g}(x)+1$. There exists $\ti{\rho}\in\R$ such that for each $z\in\R$ and $n\in\Z$ we have $|\ti g^n(z)-z-n\ti\rho|<1$, see \cite{KAT} for example. The number $\ti\rho$ is called the \emph{rotation number} of $\ti g$ and denoted $\ti{\rho}(\ti{g})$. It defines a map $\ti{\rho}:\ti{\Homeo}_+(\Ss^1)\ra\R.$ \\

We denote by $\ti\delta:\R\ra\R$ the displacement function of $\ti{g}$ where $\ti\delta(z)=\ti{g}(z)-z$ is one-periodic and lifts  where for every $\ti g\in\mathrm{Homeo}_+(\Ss^1)$
$$\ti{\rho}(\ti{g})=\int_{\Ss^1}\delta d\mu= \lim_{n\ra \infty} \frac{1}{n}\sum_{i=1}^n \delta(g^i(z)).$$
The map $\ti{\rho}$ is the unique homogeneous quasi-morphism from $\ti{\Diff}^1_+(\Ss^1)$ to $\R$ which takes the value $1$ on the translation by $1$, see \cite{GHYS2} for example. More precisely for each $\ti{f},\ti{g}\in\ti{\Homeo}_+(\Ss^1)$ it holds that $|\ti{\rho}(\ti{f})-\ti{\rho}(\ti{g})|<1$ and for each $n\in\Z$ we have $\ti{\rho}(\ti{f}^n)=n\ti{\rho}(\ti{f})$.  

Moreover, $\ti{\rho}(\ti{g})$ naturally lifts a map $\rho: \Homeo_+(\Ss^1)\ra\Tt^1$. Indeed, if we consider $g\in\Homeo_+(\Ss^1)$ and two lifts $\ti g$ and $\ti g'$ of $g$ there exists $k\in\Z$ such that $\ti g'=\ti g$ hence we have $\ti\rho (\ti g')= \ti\rho(\ti g)+k$. By the Birkhoff ergodic theorem for every $z\in\R$ and every $g$-invariant measure $\mu$ we have
$$\ti{\rho}(\ti{g})=\int_{\Ss^1}\delta d\mu.$$

Let us describe why $\ti{\rho}$ is not a morphism and only a quasi-morphism. A homeomorphism of the circle has a fixed point if and only if its rotation number is zero, see \cite{KAT} chapter $11$ for more details. Below we give an example of two homeomorphisms $\phi$ and $\psi$ of $\Ss^1$ of rotation number zero such that the composition $\phi\circ\psi$ gives us a homeomorphism as in Figure \ref{exampleCalcompo} without fixed point and so the rotation number of the composition is not equal to $0$. \\

Let us consider the two homeomorphisms of rotation number $0$ with one fixed point as in Figures \ref{exampleCal} and \ref{exampleCalcompo}.

 \begin{figure}[htp]
 \begin{center}
 \begin{tikzpicture}

\draw (0,0) -- (3,0);
\draw (3,0) -- (3,3);
\draw (3,3) -- (0,3);
\draw (0,3) -- (0,0);
\draw[dotted] (0,0) -- (3,3);

\draw[red] (0,0) ..controls +(0.5,0.5) and +(-1,-1).. (2,1);
\draw[red] (2,1) ..controls +(0.5,0.5) and +(-0.5,-0.5).. (3,3);
\draw (3/2,-0.5) node{$\phi$};

\draw (5,0) -- (8,0);
\draw (8,0) -- (8,3);
\draw (8,3) -- (5,3);
\draw (5,3) -- (5,0);
\draw[dotted] (5,0) -- (8,3);
\draw[red] (6,0) ..controls +(0.1,0.5) and +(-0.5,-0.5).. (7,2);
\draw[red] (7,2) ..controls +(0.5,0.5) and +(-0.1,0).. (8,2.5); 
\draw[red] (5,2.5) ..controls +(0.1,0) and +(-0.1,-0.5).. (6,3);
\draw (6.5,-0.5) node{$\psi$}; 

\end{tikzpicture}
\end{center}
\caption{ }
\label{exampleCal}
\end{figure}

 \begin{figure}[htp]
 \begin{center}
 \begin{tikzpicture}

\draw (5,0) -- (8,0);
\draw (8,0) -- (8,3);
\draw (8,3) -- (5,3);
\draw (5,3) -- (5,0);
\draw[dotted] (5,0) -- (8,3);
\draw[red] (6,0) ..controls +(0.1,0.5) and +(-0.5,-0.5).. (7,1.7);
\draw[red] (7,1.7) ..controls +(0.5,0.5) and +(-0.1,0).. (8,2.5); 
\draw[red] (5,2.5) ..controls +(0.1,0) and +(-0.1,-0.5).. (6,3);
\draw (6.5,-0.5) node{$\phi\circ\psi$}; 

\end{tikzpicture}
\end{center}
\caption{ }
\label{exampleCalcompo}
\end{figure}

For $g\in\Homeo^+(\Ss^1)$ there is a bijection between the lifts of $g$ to $\R$ and the isotopies from $\id$ to $g$ as follows. Let $I=(g_t)_{\tin}$ be an isotopy from $\id$ to $g$, the lifted isotopy $\ti I=(\ti g_1)_{\tin}$ of $I$ defines a unique lift $\ti g_1$ of $g$. Then for an isotopy $I$ from $\id$ to $g$, let us denote $\ti{g}$ the time-one map of the lifted isotopy $\ti{I}$ on $\R$, we can define the rotation number $\ti\rho(I)\in\R$ of $I$ to be the rotation number $\ti\rho(\ti g)$ of $\ti{g}$. If we consider $f$ a homeomorphism of the disk isotopic to the identity and $I=(f_t)_{\tin}$ an isotopy from $\id$ to $f$ then we will denote $\ti\rho(I|_{\Ss^1})\in\R$ the rotation number of the restriction of the isotopy $I$ to $\Ss^1$. If we consider another isotopy $I'$ from $\id$ to $g$ one may prove that there exists an integer $k\in\Z$ such that $I'$ is homotopic to $R^kI$ where the isotopy $R=(R_t)_{\tin}$ satisfies $R_t(z)=z\e^{2 \pi i t}$ for every $z\in\Ss^1$ and every $\tin$. We consider $\ti I$ the lifted isotopy of $I'$ and we denote $\ti g'$ its time-one map. Hence $\ti g$ and $\ti g$ are two lifts of $g$ such that $\ti g'=\ti g+k$ and $\ti\rho(\ti g')=\ti\rho(\ti g)+k$ and so the number $\ti\rho(I)$ does not depend on the choice of the isotopy in the homotopy class of $I$.\\

\textbf{Irrational pseudo rotation.}
An irrational pseudo-rotation is an area-preserving homeomorphism $f$ of $\D$ that fixes $0$ and that does not possess any other periodic point. To such a homeomorphism is associated an irrational number $\alpha\in \R/\Z\backslash\Q/\Z$, called the \textit{rotation number} of $f$, characterized by the following : every point admits $\alpha$ as a rotation number around the origin. To be more precise, choose a lift $\ti{f}$ of $f|_{\D\backslash \{0\}}$ to the universal covering space $\ti{D}=\R\times(0,1]$. There exists $\ti\alpha\in\R$ such that $\ti\alpha+\Z=\alpha$ and for every compact set $K\subset \D\backslash \{0\}$ and every $\epsilon>0$, one can find $N\geq1$ such that
$$\forall n\geq N, \  \ti{z}\in \pi^{-1}(K)\cap \ti{f}^{-n}(\pi^{-1}(K)) \Rightarrow |\frac{p_2(\ti{f}^n(\ti{z}))-p_2(\ti{z})}{n}-\ti\alpha |\leq\epsilon,$$
where $\pi :(r,\theta)\mapsto (r\cos(2\pi \theta),r\sin(2\pi\theta)$ is the covering projection and $p_2: (r,\theta)\mapsto \theta$ the projection on the second coordinate. If moreover $f$ is a $C^k$ diffeomorphism $1\leq k\leq +\infty$ we will call $f$ a $C^k$ irrational pseudo-rotation.\\

Notice that for the rotation number $\alpha$ of an irrational pseudo-rotation $f$ is equal to $\rho(f|_{\Ss^1})$.\\

One can construct irrational pseudo-rotations with the method of fast periodic approximations, presented by Anosov and Katok \cite{AK}. One may see \cite{FAT1,FAY,FAY1,HAN,ROUX6} for further developments about this method and see \cite{BEG,BEG1} for other results on irrational pseudo-rotations.

				\section{Three extensions}\label{threeExtensions}

In this section we will explain why the functions $\Cal_1$, $\ti\Cal_2$ and $\ti\Cal_3$ are well-defined and we will establish the relations between them. The full statement like the continuity or the quasi-morphism property will be proved in the next section. \\
 
		 \subsection{Action function}
		 
Let us consider $f\in \Diff^1_\omega(\D)$ and $\lambda$ a Liouville $1$-form such that $d\lambda=\omega$. The fact that $H_1(\D,\R)=0$ implies that the continuous $1$-form $f^*\lambda-\lambda$ is exact. More precisely its integral along each loop $\gamma\subset \D$ is zero. Consequently the map $(r,\theta) \mapsto \int_{\gamma_z}f^*\lambda-\lambda$ is a $C^1$ primitive of $f^*\lambda-\lambda$, equal to $0$ at the origin, where for every $z\in\D$ the path $\gamma_z : [0,1] \ra \D$ is such that $\gamma_z(t)=tz$.\\

If we suppose that $f$ is compactly supported on $\mathring{\D}$ then it is natural to consider the unique $C^1$ function $A:\D\ra\R$ that is zero near the boundary of $\D$ and that satisfies
\begin{equation}
dA=f^*\lambda-\lambda.
\end{equation} 
Without the compact support hypothesis we have the following proposition.

\begin{prop}\label{mu}
If we consider a $C^1$ function $A:\D\ra\R$ such that $dA=f^*\lambda-\lambda$ then the number
$$\int_{\Ss^1} A|_{\partial \D}d\mu$$
does not depend on the choice of $\mu$ in $\mathcal{M}(f|_{\Ss^1})$. \\
\end{prop}

\begin{proof} To prove the independence over $\mu$ there are two cases to consider.\\

$\bullet$ If there exists only one $f|_{\Ss^1}$-invariant probability measure on $\Ss^1$ the result is obvious. In this case $f|_{\Ss^1}$ is said to be uniquely ergodic.\\ 
 
$\bullet$ If $f|_{\Ss^1}$ is not uniquely ergodic then by Poincaré's theory $\rho(f|_{\Ss^1})=\frac{p}{q}+\Z$ is rational with $p\wedge q=1$. The ergodic decomposition theorem, see \cite{KAT} for example, tells us that an $f|_{\Ss^1}$ invariant measure is the barycenter of ergodic $f|_{\Ss^1}$-invariant measures. Moreover, each ergodic measure of $f|_{\Ss^1}$ is supported on a periodic orbit as follows. For $z$ a $q$-periodic point of $f|_{\Ss^1}$, we define the probability measure $\mu_z$ supported on the orbit of $z$ by 
 $$\mu_z=\frac{1}{q}\sum_{k=0}^{q-1}\delta_{f^k(z)},$$
 where $\delta_z$ is the Dirac measure on the point $z\in\Ss^1$. Hence it is sufficient to prove that $\int_\D A(f,\lambda,\mu_z) \omega$ does not depend of the choice of a periodic point $z\in\Ss^1$.\\

Let us consider two periodic points $z$ and $w$ of $f|_{\Ss^1}$. We consider an oriented path $\gamma\subset \Ss^1$ from $z$ to $w$. We compute 
 \begin{align*}
 \int_{\Ss^1} Ad\mu_z-\int_{\Ss^1} Ad\mu_w & =\frac{1}{q}\sum_{k=0}^{q-1}A(f^k(z))-A(f^k(w))\\
  &=\frac{1}{q}\sum_{k=0}^{q-1}\int_{f^k(\gamma)}dA\\
  &=\frac{1}{q}\sum_{k=0}^{q-1}\int_{f^k(\gamma)}f^*(\lambda)-\lambda\\
  &=\int_{f^{q}(\gamma)}\lambda- \int_\gamma \lambda \\
  &=0 
 \end{align*}
where the last equality is due to the fact that $f^q(\gamma)$ is a reparametrization of the path $\gamma$.
\end{proof}

Proposition \ref{mu} allows us to make a natural choice of the action function to define an extension of the Calabi invariant as follows.

\begin{theorem}\label{thm1.0}
For each $f\in \Diff^1_\omega(\D)$ we consider the unique $C^1$ function $A_f$ of $f$ such that $dA_f=f^*\lambda-\lambda$ and $\int_{\Ss^1} A_fd\mu=0$ where $\lambda$ is a Liouville form of $\omega$ and $\mu$ an $f$-invariant probability measure on $\Ss^1$. The number 
$$\Cal_1(f)=\int_\D A_f(z) \omega(z)$$
does not depend on the choice of $\lambda$ or $\mu$. 
\end{theorem}
 
\begin{proof}
The independence on the measure $\mu$ comes from Proposition \ref{mu} and it remains to prove the independence on $\lambda$.\\
 Let us consider another primitive $\lambda'$ of $\omega$. We denote $A$ and $A'$ the two functions such that $dA=f^*\lambda-\lambda$ and $dA'=f^*\lambda'-\lambda'$ and such that for each $\mu\in \mathcal{M}(f|_{\Ss^1})$ we have $\int_{\Ss^1}A d\mu=\int_{\Ss^1}A' d\mu=0$.\\
 
The $1$-form $\lambda-\lambda'$ is closed because $d\lambda-d\lambda'=\omega-\omega=0$. So there exists a smooth function $u:\D\ra\R$ such that $\lambda'=\lambda+du$. We compute
\begin{align*}
dA'&=f^*(\lambda+du)-(\lambda+du)\\
&=f^*\lambda-\lambda + d(u\circ f-u)\\
&=dA+ d(u\circ f-u).
\end{align*}
Thus there exists a constant $c$ such that 
$$A'=A+u\circ f-u +c.$$
For a measure $\mu\in\mathcal{M}(f|_{\Ss^1})$ the condition $\int_{\Ss^1} A'd\mu=0=\int_{\Ss^1} Ad\mu$ implies that 
$$\int_{\Ss^1} A'd\mu=\int_{\Ss^1} Ad\mu+\int_{Ss^1} (u\circ f|_{\Ss^1} -u)d\mu+c=\int_{\Ss^1} Ad\mu,$$
Howeover $\int_{\Ss^1} (u\circ f|_{\Ss^1} -u)d\mu=0$ since $f|_{\Ss^1}$ preserves $\mu$ we have
$$c=0.$$
Finally $f$ preserves $\omega$ hence $\int_{\D}(u\circ f-u)\omega=0$ and we can conclude that
$$\int_\D A'\omega=\int_{\D} A\omega.$$
\end{proof}

We compute the extension $\Cal_1$ of rotations of the disk.
\begin{prop}\label{rotation}
For $\theta\in\R$ the rotation $R_{\theta}$ of angle $\theta$ satisfies 
$$\Cal_1(R_{\theta})=0.$$
\end{prop}

\begin{proof}
For the Liouville form $\lambda=\frac{r^2}{2\pi}d\theta$ of $\omega$ we have $R^*_\theta\lambda -\lambda=0$ thus the action function $A$ is constant. So it is equal to $0$ and we obtain the result.
\end{proof}

\subsection{Angle function}
The following interpretation is due to Fathi in his thesis \cite{FAT} in the case of compactly supported symplectic diffeomorphisms of the unit disk. This interpretation is also developped by Ghys and Gambaudo in see \cite{GHYS}.\\

Let us consider $f\in\Diff_+^1(\D)$ and $I=(f_t)_{t\in[0,1]}$ an isotopy from $\id$ to $f$. For $x,y\in\D$ distinct we can consider the vector $v_t$ from $f_t(x)$ to $f_t(y)$ and we denote by $\Ang_I(x,y)$ the angle variation of the vector $v_t$ for $t\in[0,1]$ defined as follows.\\

We have the polar coordinates $(r,\theta)$ and a differential form
$$d\theta=\frac{udv-vdu}{u^2+v^2},$$
where $(u,v)$ are the cartesian coordinates. For every couple $(x,y)\in\D^2\backslash\Delta$ we define 
\begin{equation}\label{Angpolar}
\Ang_I(x,y)=\frac{1}{2\pi}\int_\gamma d\theta,
\end{equation}
where $\gamma:t\ra f_t(x)-f_t(y)$.\\

The function $\Ang_I$ is continuous on the complement of the diagonal of $\D\times\D$. Moreover, if $f$ is at least $C^1$ then the function $\Ang_I$ can be extended on the diagonal into a bounded function on $\D\times\D$. Indeed, we consider $K$ the compact set of triplets $(x,y,d)$ where $(x,y)\in \D\times\D$ and $d$ a half line in $\R^2$ containing $x$ and $y$ and oriented by the vector joining $x$ to $y$ if $ x\neq y$. If $x$ and $y$ are distincts, the half line $d$ is uniquely determined and $\D\times\D\backslash\Delta$ can be embedded in $K$ as a dense and open set. We define $\Ang_I(x,x,d)$ as the variation of angle of the half lines $df_t(d)$ for $\tin$. This number is well-defined and extends $\Ang_I$ into a continuous function on $K$. \\

For $\ti f=(f,\ti\phi)\in\ti\Diff^1_\omega(\D)$ and two Hamiltonian isotopies $I=(f_t)_{\tin}$ and $I'=(f'_t)_{\tin}$ from $\id$ to $f$ associated to $\tphi$. The isotopies $I'$ and $I$ are homotopic so for every couple $(x,y)\in\D^2\backslash\Delta$ we have
$$\int_\gamma d\theta=\int_{\gamma'} d\theta,$$
where $\gamma:t\mapsto f_t(x)-f_t(y)$ and $\gamma':t\mapsto f'_t(x)-f'_t(y)$. Hence,  we can define the angle function $\Ang_{\ti f}$ of $\ti f$ by
$$\Ang_{\ti f}=\Ang_I.$$

We have the following lemma.

\begin{lem}\label{choiceI}
Let us consider $\ti f=(f,\tphi)\in\ti\Diff^1_\omega(\D)$. For every $(x,y)\in\D^2\backslash\Delta$ the number $\Ang_{\ti f}(x,y)-\ti{\rho}(\tphi)$ only depends on $f$.
\end{lem}

\begin{proof}

Let us consider $I'$ another isotopy from $\id$ to $f$.\\

There exists $k\in\Z$ such that $I'$ is homotopic to $R_{2\pi}^k I$ and by definition of $\Ang_h$ given by equation \ref{angdef}  we have $\Ang_{R_{2\pi}^k I}=\Ang_{I}+k$. Moreover $I'$ is in the same homotopy class of $R_{2\pi}^k I$ and we obtain $\Ang_{I'}=\Ang_I+k$. Since the rotation number also satisfies $\ti\rho(I'|_{\Ss^1})=\ti\rho(I|_{\Ss^1})+k$, the result follows.
\end{proof}

Lemma \ref{choiceI} allows us to extend the Calabi invariant on the lifted group $\ti{\Diff}^1_\omega(\D)$ as follows.
\begin{theorem}\label{Cal2}
Let us consider $\ti f=(\ti f,\tphi)\in\ti{\Diff}^1_\omega(\D)$. The number
$$\ti{\Cal}_2(\ti f)=\int_{\D^2\backslash\Delta} \Ang_{\ti f}(x,y)\omega(x)\omega(y),$$
defines a morphism $\ti{\Cal}_2:\ti{\Diff}^1_\omega(\D)\ra\R$ and induces a morphism on $\Diff^1_\omega(\D)$ defined by
$$\Cal_2(f)=\ti{\Cal}_2(\ti f)+\Z,$$
where $\ti f\in\ti\Diff^1_\omega(\D)$ is a lift of $f$.
\end{theorem}

\begin{proof}
First, $\ti\Cal$ is well-defined since the angle function $\Ang_{\ti f}$ is integrable on $\D^2\backslash\Delta$.\\

Let us consider $\ti f=(f,\tphi)$ and $\ti g=(g,\tphi')$ two elements of $\ti{\Diff}^1_\omega(\D)$ and two isotopies $I=(f_t)_{\tin}\in[I]$ from $\id$ to $f$ associated to $\tphi$ and $I'=(g_t)_{\tin}$ from $\id$ to $g$ associated to $\tphi'$. We consider the concatenation $I\cdot I'$ of the isotopy $I$ and $I'$ which gives an isotopy from $\id$ to $f\circ g$ associated to $\tphi\circ\tphi'$ and we define the element $\ti f\circ\ti g=(f\circ g,\tphi\circ\tphi')\in \ti{\Diff}^1_\omega(\D)$. For each $(x,y)\in\D^2\backslash\Delta$ we have
$$\Ang_{I\cdot I'}(x,y)=\Ang_{I'}(x,y)+\Ang_{I}(g(x),g(y)).$$
Hence we obtain
$$\Ang_{\ti f\circ\ti g}(x,y)=\Ang_{\ti g}(x,y)+\Ang_{\ti f}(g(x),g(y)).$$
We integrate the previous equality and since $g$ preserves $\omega$ we deduce that $\ti{\Cal}_2$ is a morphism from $\ti{\Diff}^1_\omega(\D)$ to $\R$.\\

Moreover, Lemma \ref{choiceI} assures that $\ti{\Cal}_2$ induces the morphism $\Cal_2$ from $\Diff^1_\omega(\D)$ to $\Tt^1$.
\end{proof}

Notice that the morphisms $\ti\Cal_2$ and $\Cal_2$ satisfy the following commutative diagram 
$$\xymatrix{
    \ti{\Diff}^1_\omega(\D) \ar[r] \ar[d]_{\ti{\Cal}_2}  & \Diff^1_\omega(\D) \ar[d]^{\Cal_2}  \\
    \R \ar[r]  & \Tt^1 \\
  }$$ 
where the horizontal arrows are the covering maps.\\

This interpretation allows us to generalize the definition to other invariant measures of the disk. Let us consider $\ti f=(f,\tphi)\in\ti\Diff^1(\D)$ and an isotopy $I$ from $\id$ to $f$ associated to $\ti\phi$. We consider a probability measure $\mu$ on $\D$ without atom which is $f$-invariant. We define the number $\ti{\mathcal{C}}_\mu(I)$ by 
$$\widetilde{\mathcal{C}}_\mu(\ti f)=\int\int_{\D^2\backslash\Delta} \Ang_{\ti f}(x,y)d\mu(x) d\mu(y).$$

By Lemma \ref{choiceI} we obtain the following corollary.
\begin{coro}
Let us consider $\ti f=(f,\tphi)\in\ti\Diff^1_\omega(\D)$. For every $(x,y)\in\D^2\backslash\Delta$ the number $\ti{\mathcal{C}}_\mu(\ti f)-\ti{\rho}(\ti\phi)$ only depends on $f$.
\end{coro}

Birkhoff ergodic theorem gives another way to compute $\ti{\mathcal{C}}_\mu(\ti f)$ for $\ti f=(f,\tphi)\in\ti \Diff^1(\D)$. Let us consider an isotopy $I=(f_t)_{\tin}$ from $\id$ to $f$ associated to $\tphi$. For $(x,y)\in\D\times\D\backslash\Delta$ we have 
\begin{equation}
\Ang_{I^n}(x,y)=\Ang_I(x,y)+\Ang_I(f(x),f(y))+...+\Ang_I(f^{n-1}(x),f^{n-1}(x)).
\end{equation}
The function $\Ang_I$ is bounded so the function 
$$\widehat{\Ang}_I(x,y)=\lim_{n\ra\infty}\frac{1}{n}\Ang_{I^n}(x,y),$$
is defined $\mu\times\mu$ almost everywhere and depends only on the homotopy class of $I$. Hence we can define $\widehat\Ang_{\ti f}=\widehat\Ang_I$ . Thus we obtain the following equality
\begin{equation}\label{Cmu}
\ti{\mathcal{C}}_\mu(\ti f)=\int\int_{\D \times\D} \widehat{\Ang}_{\ti f}(x,y)d\mu(x) d\mu(y).
\end{equation}

We state the proposition of topological invariance, see \cite{GHYS}.

\begin{prop}\label{conjugaison}
Let us consider two probability measures $\mu_1$ and $\mu_2$ of $\D$ without atom and two compactly supported elements of $\Diff_{\mu_1}^1(\D)$ and $ \Diff_{\mu_2}^1(\D)$ denoted $\phi_1$ and $\phi_2$ such that there exists a homeomorphism $h\in \Diff_+^0(\D)$ satisfying $\phi_2=h\circ\phi_1\circ h^{-1}$ and $h_*(\mu_1)=\mu_2$. We have that
$$\mathcal{C}_{\mu_1}(\phi_1)=\mathcal{C}_{\mu_2}(\phi_2).$$
\end{prop}

For a probability measure $\mu$ of the disk, there is the equivalent result to extend the invariant $\mathcal{C}_\mu$.

\begin{theorem}
Let us consider an element $\ti f\in\ti{\Diff}^1_\mu(\D)$. The number
$$\ti{\mathcal{C}}_\mu(\ti f)=\int_{\D^2\backslash\Delta} \Ang_{\ti f}(x,y)d\mu(x)d\mu(y),$$
defines a morphism $\ti{\mathcal{C}}_\mu:\ti{\Diff}^1_\mu(\D)\ra\R$ which induces a morphism $\mathcal{C}_\mu:\Diff^1_\mu(\D)\ra \Tt^1$ defined for every $f\in\Diff^1_\mu(\D)$ by
$$\mathcal{C}_\mu(f)=\ti{\mathcal{C}}_\mu(\ti f)+\Z,$$
where $\ti f\in\ti\Diff^1_\mu(\D)$ is a lift of $f$.
\end{theorem}

The proof of the previous theorem is basically the same as Theorem \ref{Cal2} and if we consider the Lesbegue measure $\Leb$ then we have 
$$\ti{\mathcal{C}}_\Leb=\ti{\Cal}_2.$$

We have the following computation in the case of the rotations.
\begin{lem}\label{rotation}
For $\theta\in\R$ we consider $\ti R_\theta=(R_\theta,\ti{r})\in\ti\Diff^1_\omega(\D)$ where $R_\theta$ is the rotation $\D\ra\D$ of angle $\theta$. We have
$$\ti{\Cal}_2(\ti R)=\ti{\rho}(\ti{r}).$$
\end{lem}

\begin{proof}
Let us consider $R=(R_t)_{\tin}$ the isotopy from $\id$ to $R_\theta$ given in section \ref{preliminaries}. For a couple $(x,y)\in\D\times\D\backslash \Delta$ we consider the complex $z=x-y$ and we have for each $\tin$ $R_t(z)=z\e^{it\theta}$ and we can compute $\Ang_R(x,y)=\theta$. By integration on $\D\times\D\backslash \Delta$ we obtain 
$$\ti\Cal_2(\ti R_\theta)=\theta=\ti\rho(\ti r).$$
\end{proof}

				\subsection{Hamiltonian function}

In this section, the goal is to state the construction of the Calabi invariant given by equation \ref{Calintro} in the case of compactly supported diffeomorphisms of the disk. This construction leads to Theorem \ref{thm3} and we explain the definition of $\ti{\Cal}_3$ given by this theorem but we refer to the next section for the proofs of certain results.

Let us consider $f\in\Diff^1_\omega(\D)$ and a Hamiltonian isotopy $I=(f_t)_{\tin}$ from $\id$ to $f$. We consider the Hamiltonian function $(H_t)_{t\in\R}$ which induces the isotopy $I$. We denote $(X_t)_{t\in\R}$ the associated vector field. We have that for every $t\in\R$, $X_t$ is tangent to $\Ss^1$. So each $H_t$ is constant on $\Ss^1$ and we can consider $(H_t)_{t\in\R}$ the associated Hamiltonian function such that$$H_t|_{\Ss^1}=0.$$

We have the following lemma.
\begin{lem}\label{hamchoice}
The integral
$$2\int_{z\in\D} \int_0^1H_t(z)\omega(z)dt-\ti\rho(I|_{\Ss^1}),$$
depends only on $f$.
\end{lem}

\begin{proof}
The result will be a corollary of Theorem \ref{thmlink}. 
\end{proof}

\begin{theorem}
Let us consider an element $\ti f=(f,\ti\phi)\in\ti{\Diff}^1_\omega(\D)$ and a Hamiltonian function $H:\Ss^1\times \D\ra\R$ of $f$ which induces the flow $(\phi_t)_{\tin}$ such that the lift of $\phi_1|_{\Ss^1}$ is equal to $\ti{\phi}$ and such that $H_t$ is equal to $0$ on $\Ss^1$ for every $t\in\R$. The number
$$\ti{\Cal}_3(\ti f)=\int_0^1\int_\D H_t(z)\omega(z)dt,$$
does not depend on the choice of $H$. Moreover the map $\ti{\Cal}_3: \ti{\Diff}^1_\omega(\D) \ra\R$ is a morphism and $\ti\Cal(\ti f)+\Z$ depends only on $f$. It induces a morphism 
$$\Cal_3(f)=\int_0^1\int_\D H_t(z)\omega(z)dt+\Z,$$
defined on $\Diff^1_\omega(\D)$.
\end{theorem}

The proof comes from the equality between $\ti{\Cal}_2$ and $\ti{\Cal}_3$ which will be proven in the next section. Moreover, the definition of $\Cal_3$ comes from Lemma \ref{hamchoice} and we obtain the following commutative diagram where the horizontal arrows are the universal covering maps.
$$\xymatrix{
    \ti{\Diff}^1_\omega(\D) \ar[r] \ar[d]_{\ti{\Cal}_3}  & \Diff^1_\omega(\D) \ar[d]^{\Cal_3}  \\
    \R \ar[r]  & \Tt^1 \\
  }$$

			\section{Proof of Theorem \ref{thmlink}.}\label{Proof}

In this section, we prove Theorem \ref{thmlink}.

\begin{theorem}\label{equality}
The morphisms $\ti{\Cal}_2$ and $\ti{\Cal}_3$ are equal. For $\ti f=(f,\ti{\phi})\in \ti{\Diff}^1_\omega(\D)$ we have the following equality
$$\ti{\Cal}_2(\ti f)=\Cal_1(f)+\ti{\rho}(\ti{\phi}).$$
Moreover $\Cal_1$, $\ti{\Cal}_2$ and $\ti{\Cal}_3$ are continuous in the $C^1$ topology.
\end{theorem}

We separate the proof into two subsections, in the first one we establish the links between the previous definitions then we prove the continuity of $\ti\Cal_2$ and $\ti\Cal_3$.

\subsection{Equality between $\ti\Cal_2$ and $\ti\Cal_3$.}

\begin{prop}
The morphisms $\ti{\Cal}_2$ and $\ti{\Cal}_3$ are equal.
\end{prop}

\begin{proof}  The proof is essentially the same as in \cite{SHE}, the only difference is that our symplectic form is normalized and the Hamiltonian diffeomorphisms that we consider is not compactly supported in the open unit disk. Nevertheless, we verify that the proof is still relevant in our case.\\

Let us consider $\ti f=(f,\ti\phi)\in\ti{\Diff}^1_\omega(\D)$ and a Hamiltonian isotopy $I=(f_t)_{t\in[0,1]}$ from $\id$ to $f$ associated to $\tphi$. For the proof we will give a definition of the angle function $\Ang_I$ in the complex coordinates as follows. We define a $1$-form $\alpha$ by 
$$\alpha=\frac{1}{2\pi}\frac{d(z_1-z_2)}{z_1-z_2}.$$
The imaginary part satisfies
$$d\theta=2\pi\text{Im}(\alpha),$$
where $\theta$ is the angle coordinate in the radial coordinates. For an element $Z=(z_1,z_2)\in\D^2\backslash \Delta$ of we consider the curve $I_Z\subset \D\times\D\backslash\Delta$ defined by 
$$t\mapsto I_Z(t)=(f_t(z_1),f_t(z_2)),$$
for each $t\in[0,1]$ and that for every element $Z=(z_1,z_2)\in \D\times\D\backslash\Delta(\D)$ we have
\begin{equation}\label{angdef}
\Ang_I(z_1,z_2)=\frac{1}{2\pi}\int_{I_Z}d\theta.
\end{equation} 

Let us consider the Hamiltonian $(H_t)_{\tin}$ which induces the flow of the isotopy $I$ and which is equal to $0$ on the boundary of $\D$. We consider the symplectic form $\omega=\frac{i}{2\pi}dz\wedge\overline{z}$ written in the complex coordinates on $\D$. We define $\xi_t=dz(X_t)$ and then it satisfies 
$$i_{X_t}\left(\frac{i}{2\pi}dz\wedge\overline{z}\right)=\frac{i}{2\pi}\xi_td\overline{z}-\frac{i}{2\pi}\overline{\xi}_tdz.$$

By definition
$$dH_t=-\frac{\partial H_t}{\partial\overline{z}}d\overline{z}-\frac{\partial H_t}{\partial z}dz,$$

so we have
\begin{equation}\label{xi}
\xi_t=-2i\pi\frac{\partial H_t}{\partial\overline{z}}.
\end{equation}

We compute the integral of the angle function
\begin{align*}
\int_{\D\times\D\backslash\Delta}\mathrm{Ang}_I(z_1,z_2)\omega(z_1)\omega(z_2)&=\int_{\D\times\D\backslash\Delta}\int_{I_{(z_1,z_2)}}\frac{1}{2\pi}d\theta \ \omega(z_1)\omega(z_2)\\
&=\mathrm{Im}\left(\int_{\D\times\D\backslash\Delta}\int_{I_{(z_1,z_2)}}\alpha \ \omega(z_1)\omega(z_2)\right).
\end{align*}

The following computation is well-inspired by the proof in \cite{SHE}.
\begin{align*}
\int_{\D\times\D\backslash\Delta}\int_{I_{(z_1,z_2)}}\alpha \ \omega(z_1)\omega(z_2)&=\frac{1}{2\pi}\int_{\D\times\D\backslash\Delta}\int_{I_{(z_1,z_2)}} \frac{d(z_1-z_2)}{z_1-z_2}\omega(z_1)\omega(z_2)\\
&=\frac{1}{2\pi}\int_{\D\times\D\backslash\Delta}\int_{t=0}^1\frac{\xi_t(f_t(z_1))-\xi_t(f_t(z_2))}{f_t(z_1)-f_t(z_2)}dt\omega(z_1)\omega(z_2),\\
&=\frac{1}{2\pi}\int_{t=0}^1\int_{\D\times\D\backslash\Delta}\frac{\xi_t(f_t(z_1))-\xi_t(f_t(z_2))}{f_t(z_1)-f_t(z_2)}\omega(z_1)\omega(z_2)dt,\\
&=2\times\frac{1}{2\pi}\int_{t=0}^1\int_{z_2\in\D}\int_{z_1\in\D\backslash\{z_2\}}\frac{\xi_t(z_1)}{z_1-z_2}\omega(z_1)\omega(z_2)dt\\
&=\frac{1}{\pi}\int_0^1\int_\D \int_{\D\backslash\{z_2\}}-2i\pi\frac{\partial H_t}{\partial\overline{z}}\frac{i}{2\pi} \frac{dz_1\wedge d\overline{z_1}}{z_1-z_2}\omega(z_2)dt\\
&=2i \int_0^1\int_\D \int_{\D\backslash\{z_2\}}\frac{1}{2i\pi}\frac{\partial H_t}{\partial\overline{z}} \frac{dz_1\wedge d\overline{z_1}}{z_1-z_2}\omega(z_2)dt.
\end{align*}

The third equality is obtained by Fubini because the integral is absolutely integrable by Lemma \ref{Fubini}. The fourth equality is due to the absolutely integrability of both terms. We established the penultimate with equation \ref{xi} and the definition of $\omega$.\\

We use the Cauchy formula for smooth functions (see \cite{LARS}). For any $C^1$-function $g:\D\ra\C$, we have 
$$g(w)=\frac{1}{2i\pi}\int_{\Ss^1} \frac{g(z)}{z-w}dz+\frac{1}{2i\pi}\int_\D \frac{\partial f}{\partial\overline{z}}\frac{dz\wedge d\overline{z}}{z-w}.$$
Moreover $H_t$ is equal to zero on the boundary $\Ss^1$ and we have
$$\int_{\D\times\D\backslash\Delta}\int_{I_{(z_1,z_2)}}\alpha \ \omega(z_1)\omega(z_2)=2i\int_0^1\int_{\D}H_t(z_2)\omega(z_2)dt.$$
It leads to 
$$\int_{\D\times\D\backslash\Delta}\Ang_I(z_1,z_2)\omega(z_1)\omega(z_2)=2\int_0^1\int_\D H_t(z)\omega(z)dt.$$
To obtain the result it remains to prove the absolute integrability we used in the computation.

\begin{lem}\label{Fubini}
We have the following inequality 
$$\int_{\D\times\D\backslash\Delta}\int_{t=0}^1\left|\frac{\xi_t(f_t(z_1))-\xi_t(f_t(z_2))}{f_t(z_1)-f_t(z_2)}\right|\omega(z_1)\omega(z_2)dt<\infty.$$
\end{lem}

\begin{proof}
The total measure of $\D\times\D\backslash\Delta$ for $\omega$ and $[0,1]$ for the Lebesgue measure is finite so by Tonnelli's theorem it is sufficient to have the following inequalities 
\begin{align*}
\int_{t=0}^1\int_{\D\times\D\backslash\Delta}\left|\frac{\xi_t(f_t(z_1))-\xi_t(f_t(z_2))}{f_t(z_1)-f_t(z_2)}\right|\omega(z_1)\omega(z_2)dt&=\int_{t=0}^1\int_{\D\times\D\backslash\Delta}\left|\frac{\xi_t(z_1)-\xi_t(z_2)}{z_1-z_2}\right|\omega(z_1)\omega(z_2)dt\\
&\leq 2\int_{t=0}^1\int_{z_1\in\D}|\xi_t(z_1)|\int_{z_2\in\D\backslash\{z_1\}}\frac{1}{|z_1-z_2|}\omega(z_1)\omega(z_2)dt\\
&\leq 8\pi\int_{t=0}^1\int_{z_1\in\D}|\xi_t(z_1)|\omega(z_1)dt\\
&<\infty.
\end{align*}

To prove the second last inequality one may prove that
$$\int_{z_2\in\D\backslash\{z_1\}}\frac{1}{|z_1-z_2|}\omega(z_2)\leq 4\pi.$$
\end{proof}
\end{proof}

\begin{rem}
The number $\ti{\Cal}_2(f,\tphi)$ does not depend on the choice of the isotopy in the homotopy class of $I$, we obtain the same result for the construction of $\ti{\Cal}_3(f,\tphi)$ which completes the proof of Lemma \ref{hamchoice}. 
\end{rem}

\begin{prop}
For each element $\ti f=(f,\ti{\phi})\in \ti{\Diff}^1_\omega(\D)$ we have 
$$\ti{\Cal}_3(\ti f)=\Cal_1(f)+\ti{\rho}(\ti{\phi}).$$
\end{prop}

\begin{proof}
Let us consider an element $\ti f=(f,\ti{\phi})\in \ti{\Diff}^1_\omega(\D)$ and a Hamiltonian isotopy $I=(f_t)_{\tin}$ from $\id$ to $f$ associated to $\tphi$. There exists a unique Hamiltonian function $(H_t)_{t\in \R}$ which induces the isotopy $I$ and such that $H_t$ is zero on the boundary $\Ss^1$ of $\D$ for each $t\in\R$.\\
 
We know that $\Cal_1$ does not depend on the choice of the primitive of $\omega$. We consider the Liouville $1$-form $\lambda=\frac{r^2}{2\pi}d\theta$ in the radial coordinates. We consider a probability measure $\mu\in\mathcal{M}(f|_{\Ss^1})$. \\

We describe the link between the action function of the first definition and the Hamiltonian of the third definition. We consider a $C^1$ family of functions $(A_t)_{\tin}$, where $A_t:\D\ra\R$ satisfies for each $\tin$
$$dA_t=f_t^*\lambda-\lambda,$$
and such that the map $A_1$ is equal to $A(f,\lambda,\mu)$. So the isotopy $(A_t)_{\tin}$ satisfies 
\begin{align*}
d\dot A_t&=\frac{d}{dt}(f_t^*\lambda)\\
&=f_t^*\mathcal{L}_{X_t}\\
&=f_t^*(i_{X_t}(d\lambda)+d(\lambda(X_t)))\\
&=d(H_t\circ f_t+ \lambda(X_t)\circ f_t).
\end{align*}
Then there exists a constant $c_t:[0,1]\ra\R$ such that 
$$\dot A_t=H_t\circ f_t+ \lambda(X_t)\circ f_t +c_t,$$
and the map $A:\D\ra\R$ satisfies for each $z\in\D$ 
$$A_1(z)=\int_0^1(H_t+i_{X_t}\lambda)(f_t(z))dt +\int_0^1c_tdt.$$
We denote by $C$ the constant $\int_0^1c_tdt$. Since the restriction of $\lambda$ to $\Ss^1$ is equal to $\frac{1}{2\pi}d\theta$ then for every $z\in\Ss^1$ we have
$$\int_0^1i_{X_t}\lambda(f_t(z))dt=\frac{1}{2\pi}\int_0^1 d\theta(\frac{\partial}{\partial t}f_t(z))dt.$$
Notice that the last integral is equal to the displacement function $\delta:\R\ra\R$ of $\tphi$.\\

Moreover, the rotation number $\ti{\rho}(\tphi)$ of the isotopy $I$ satisfies for each $z\in\Ss^1$
$$\ti{\rho}(\tphi)= \underset{n\ra \infty}{\text{lim}}\frac{1}{n}\sum_{k=0}^{n-1}\delta(\tphi^k(z)).$$
The map $z\mapsto \delta(z)$ is $\mu$ integrable and the Birkhoff ergodic theorem gives us
$$\int_{\Ss^1}\ti{\rho}(\tphi)d\mu(z)=\int_{\Ss^1}\delta(z)d\mu(z).$$
We obtain 
$$\int_{\Ss^1}\int_0^1i_{X_t}\lambda(f_t(z))dtd\mu(z)=\int_{\Ss^1}\ti{\rho}(\tphi)d\mu(z)=\ti\rho(I|_{\Ss^1}).$$
Moreover, the Hamiltonian $H_t$ is equal to zero on $\Ss^1$. So if $z\in\Ss^1$ it holds that $A_1(z)=\delta(z)+C$ and consequently
$$\int_{\Ss^1}A_1(z)d\mu(z)=C+\ti{\rho}(\tphi).$$
So the condition on $A$ implies that 
$$C=-\ti{\rho}(\tphi).$$
Thus 
\begin{align*}
\int_\D A(z)\omega(z)&=\int_\D\int_0^1(H_t+i_{X_t}\lambda)(f_t(z))dt\omega(z)-\ti{\rho}(\tphi) \\
&=\int_\D\int_0^1H_t(f_t(z))dt\omega(z)+\int_\D\int_0^1i_{X_t}(\lambda)(f_t(z))dt\omega(z)-\ti{\rho}(\tphi).
\end{align*}
We compute $\int_\D\int_0^1i_{X_t}(\lambda)(f_t(z))dt\omega(z)$. Each $3$-form is zero on the disk so we have 
\begin{align*}
0&=i_{X_t}(\lambda\wedge \omega)\\
&=i_{X_t}(\lambda) \omega -\lambda\wedge i_{X_t}(\omega)\\
&=i_{X_t}(\lambda) \omega - \lambda \wedge dH_t\\
&=i_{X_t}(\lambda) \omega +dH_t\wedge \lambda\\
&=i_{X_t}(\lambda) \omega +d(H_t\lambda)-H_t\omega.\\
\end{align*}
We deduce that 
\begin{align*}
\int_\D\int_0^1i_{X_t}(\lambda)(f_t(z))dt\omega(z)&=\int_\D\int_0^1 (H_t\omega-d(H_t\lambda))dt\\
&=\int_\D\int_0^1 H_t\omega dt \omega dt-\int_0^1\int_{\Ss^1}H_t\lambda dt\\
&=\int_\D\int_0^1 H_t \omega dt,
\end{align*}
where the last equality is due to the fact that $f_t$ preserves $\omega$. Moreover $H_t$ is equal to zero on the boundary $\Ss^1$. We obtain 
$$\int_\D A(z)\omega(z)=2\int_\D\int_0^1 H_t(z) \omega(z)dt-\ti{\rho}(\tphi).$$
\end{proof}

We know that $\ti\rho$ is a homogeneous quasi-morphism, it gives us the following corollary. 

\begin{coro}
The map $\Cal_1:\Diff^1_\omega(\D)\ra \R$ is a homogeneous quasi-morphism.
\end{coro}

\begin{proof}
The result is straightforward because $\Cal_1$ is equal to the sum of a morphism and a homogeneous quasi-morphism.
\end{proof}

Notice that Lemma \ref{rotation} ensures that the morphisms $\ti\Cal$ (resp. $\Cal$) is not zero, then its kernel is a normal non trivial subgroup of $\ti\Diff^1_\omega(\D)$ (resp. $\Diff^1_\omega(\D)$ and we obtain the following corollary. 

\begin{coro}
The groups $\ti{\Diff}^1_\omega(\D)$ and $\Diff^1_\omega(\D)$ are not perfect.
\end{coro}

					\subsection{Continuity of $\ti\Cal$.}
					
For every continuous map $f$ from $\D$ to $\C$ we set $||f||_\infty=\max_{x\in \D}|f(x)|$.\\
We denote $d_0$ the distance between two maps $f$ and $g$ of $\Diff^0(\D)$ defined by
$$d_0(f,g)=\max (||f-g||_\infty, ||f^{-1}-g^{-1}||_\infty).$$
We denote $d_1$ the distance between two maps $f$ and $g$ of $\Diff^1(\D)$ defined by
$$d_1(f,g)=\max(d_0(f,g), ||Df-Dg||_\infty,||Df^{-1}-Dg^{-1}||_\infty),$$
where for every $C^1$ diffeomorphism $f$ of $\D$, $||Df||_\infty =\max_{x\in\D} ||D_xf||$.\\

The distances $d_0$ and $d_1$ define naturally two distances, denoted $\ti d_0$ and $\ti d_1$, on $\ti{\Diff}_\omega^1(\D)$ defined as follows. Let us consider $\ti f=(f,\ti{\phi})$ and $\ti g=(g,\ti{\psi})$ in $\ti{\Diff}_\omega^1(\D)$, we have
\begin{align*}
\ti d_0(\ti f,\ti g)=\max(d_0(f,g), ||\ti\phi-\ti\psi||_\infty,||\ti\phi^{-1}-\ti\psi^{-1}||_\infty),\\
\ti d_1(\ti f,\ti g)=\max(d_1(f,g), ||\ti\phi-\ti\psi||_\infty,||\ti\phi^{-1}-\ti\psi^{-1}||_\infty).
\end{align*}

We denote $\ti\id=(\id_\D,\id_\R)\in\ti\Diff^1_\omega(\D)$. In this section we prove the following result.
\begin{theorem}\label{calabiC0}
The map $\ti\Diff_\omega^1(\D) \ra \R$ is continuous in the $C^1$ topology.
\end{theorem}

We need some results about the angle function.
\begin{lem}\label{lemAngC0}
Let us consider $\ti f=(f,\tphi)\in\Diff_+^1(\D)$ such that $\ti d_1(\ti f,\ti\id)\leq\epsilon\leq 1/2$, then for every $(x,y)\in\D^2\backslash\Delta$, it holds that 
$$|\cos(2\pi\Ang_{\ti f}(x,y))-1|\leq 2\epsilon.$$
\end{lem}

\begin{proof}[Proof of Lemma \ref{lemAngC0}]
The proof is a simple computation. Let us consider $x,y\in\D$ such that $x\neq y$. One can write $f=\id+h$ where $||h||_\infty\leq \epsilon$ and $||Dh||_\infty\leq\epsilon$. By the mean theorem we have 
\begin{equation}\label{epsilon}
\left|\frac{h(y)-h(x)}{y-x}\right|\leq\epsilon .
\end{equation}
We have 
$$\cos(2\pi\Ang_{\ti f}(x,y))=\left\langle \frac{f(y)-f(x)}{|f(y)-f(x)|} \big| \frac{y-x}{|y-x|}\right\rangle,$$
where $\langle . | . \rangle$ is the canonical scalar product on $\R^2$. We compute
\begin{align*}
|\cos(2\pi\Ang_{\ti f}(x,y))-1|&=\left|\left\langle \frac{f(y)-f(x)}{|f(y)-f(x)|}-\frac{y-x}{|y-x|} \big| \frac{y-x}{|y-x|}\right\rangle\right|\\
&\leq \left| \frac{f(y)-f(x)}{|f(y)-f(x)|}-\frac{y-x}{|y-x|} \right|.
\end{align*}

We compute 
\begin{align*}
\left| \frac{f(y)-f(x)}{|f(y)-f(x)|}-\frac{y-x}{|y-x|} \right|&\leq \left| \frac{f(y)-f(x)-(y-x)}{|y-x|}\right|+|f(y)-f(x)|\left| \frac{1}{|f(y)-f(x)|}- \frac{1}{|y-x|}\right| \\
&\leq\left| \frac{h(y)-h(x)}{|y-x|}\right|+\left|\frac{|y-x|-|f(y)-f(x)|}{|y-x|}\right|\\
&\leq 2\left| \frac{h(y)-h(x)}{y-x}\right|\\
&\leq2\epsilon.
\end{align*}
\end{proof}

From Lemma \ref{lemAngC0}, we deduce the following result.

\begin{coro}\label{AngC0}
Let us consider $\ti f\in \ti{\Diff}^1_\omega(\D)$ such that $d_1(\ti f,\ti \id)\leq\epsilon\leq 1/2$. The angle function satisfies 
$$||\Ang_{\ti f}||_\infty\leq \sqrt\epsilon/\pi.$$
\end{coro}
 
\begin{proof}
For every couple $(x,y)\in\D^2\backslash\Delta$ there exists a unique $k\in\Z$ such that $\Ang_{\ti f}(x,y)-k\in[-1/2,1/2)$. So by Lemma \ref{lemAngC0} we have
$$1\geq \cos(|2\pi\Ang_{\ti f}(x,y)-k|)\geq 1-2\epsilon\geq 0.$$
The function arccos is decreasing so we obtain 
$$0\leq \arccos(\cos(|2\pi\Ang_{\ti f}(x,y)-k|))\leq \arccos(1-2\epsilon).$$
Moreover the function arccos is defined on $[0,1]$ and of class $C^1$ on $[0,1)$ such that for every $x\in(0,1]$ we have
$$(\arccos(1-x))'=\frac{1}{\sqrt{2x-x^2}}\leq \frac{1}{\sqrt x}.$$
We obtain that for every $x\in[0,1]$ we have 
$$\arccos(1-x)\leq 2\sqrt{x}.$$
Hence we have 
$$2\pi|\Ang_{\ti f}(x,y)-k|\leq 2\sqrt{2\epsilon}.$$
And so
$$|\Ang_{\ti f}(x,y)-k|\leq \frac{\sqrt{2\epsilon}}{\pi}<1/2.$$

Moreover $\D^2\backslash\Delta$ is path connected. Indeed, let us prove that every couple $(x,y)\in \D^2\backslash \Delta$ is connected to $((0,0),(1,0))$ by a path as follows. We set $d$ the line of $\D^2$ passing through $x$ and $y$. The line $d$ intersects $\Ss^1$ in two points which we denote $\hat{x}$ and $\hat y$ such that $\hat x$ is closer to $x$ than $y$ and $\hat y$ is closer to $y$ than $x$ as in figure \ref{connect}

\begin{figure}[htp]
 \begin{center}
 \begin{tikzpicture}
 
 \draw (0,0) circle (1) ;
 \draw (0.3,0.3) node{$\bullet$};
 \draw[below] (0.3,0.3) node{$x$};
  \draw (-0.3,0.3) node{$\bullet$};
    \draw[below] (-0.3,0.3) node{$y$};
   \draw (0.95,0.3) node{$\bullet$};
   \draw[right] (0.95,0.3) node{$\hat x$};
      \draw (-0.95,0.3) node{$\bullet$};
   \draw[left] (-0.95,0.3) node{$\hat y$};
   
   \draw (-0.95,0.3) -- (0.95,0.3);
   \draw[below] (-0.7,0.3) node{$d$};
 \end{tikzpicture}
\end{center}
\caption{ }
\label{connect}
\end{figure}

Let us consider the path $\gamma_y:[0,1]\ra \D$ defined by $\gamma_y(t)=t(\hat y-y)+y$ from $y$ to $\hat y$. The path $\Gamma_y: t\ra (x,\gamma_y(t))$ defined on $[0,1]$ sends the couple $(x,y)$ to $(x,\hat y)$.

Let us consider the path $\gamma_x:[0,1]\ra \D$ defined by $\gamma_x(t)=(1-t)x$ from $x$ to $(0,0)$. The path $\Gamma_x: t\ra (\gamma_x(t),\hat y)$ defined on $[0,1]$ sends the couple $(x,\hat y)$ to $(0,\hat y)$.

Now we consider $R_\alpha$ the rotation of $\D$ of angle $\alpha=arg(\hat y)$. The rotation $\R_\alpha^{-1}$ sends $\hat y$ to $(1,0)$. We denote $(R_t)_{\tin}$ the isotopy from $\id$ to $R_\alpha$ such that for every $\tin$ $R_t$ is the rotation of angle $t\alpha$.

Hence the composition of the path $\Gamma_y$, $\Gamma_x$ and $t\ra((0,0),R_t^{-1}(\hat y))$ sends $(x,y)$ to $((0,0),(0,1))$.\\

Moreover $\Ang_{\ti f}$ is continuous on $\D^2\backslash \Delta$ we deduce from the last inequality that $k$ does not depend on the choice of $(x,y)$. The fact that $\ti d_1(\ti f,\ti\id)\leq \epsilon\leq1/2$ implies that $k=0$ and we obtain that for every $(x,y)\in \D^2\backslash \Delta$
$$|\Ang_{\ti f}(x,y)|\leq \sqrt\epsilon/\pi.$$
\end{proof}

We now prove the continuity of $\Cal_1$ for the $C^1$ topology.
\begin{proof}[Proof of theorem \ref{calabiC0}]
By Theorem \ref{Cal2} we know that $\ti{\Cal}$ is a group morphism. So it is sufficient to prove the continuity at the identity. Let us consider $\ti f=(f,\tphi)\in \ti{\Diff}_\omega^1(\D)$ such that $\ti d_1(\ti f,\ti\id)\leq\epsilon\leq 1/2$. By Corollary \ref{AngC0} we have for every couple $(x,y)\in\D^2\backslash\Delta$
$$|\Ang_{\ti f}(x,y)|\leq \sqrt\epsilon/\pi.$$
By integration on $\D^2\backslash\Delta$ we obtain that
$$|\ti\Cal(\ti f)|\leq \frac{\sqrt{2\epsilon}}{\pi}.$$
Hence $\ti\Cal$ is continuous at the identity.
\end{proof}

Moreover, it is well-known that the rotation number $\ti\rho:\ti\Homeo^+(\Ss^1)\ra \R$ is continuous and we deduce from Theorem \ref{equality} the following corollary.

\begin{coro}\label{ContinuityCAL1}
The map $\Cal_1:\Diff^1_\omega(\D)\ra\R$ is continuous in the $C^1$ topology.
\end{coro}

Let us prove that the Calabi is not continuous in the $C^0$ topology.

\begin{prop}\label{notcontinuous}
The morphism $\ti\Cal$ is not continuous in the $C^0$ topology.
\end{prop}

We give a counterexample which also prove that the Calabi invariant defined in the introduction is also not continuous in the $C^0$ topology, this counterexample can be find in \cite{GHYS}

\begin{proof}
Let us consider a sequence $(h_n)_{n\geq 1}$ of smooth functions $h_n:[0,1]\ra \R$ such that 
\begin{enumerate}
\item $h_n$ is constant near the origin,
\item $h_n(r)$ is zero for $r>1/n$,
\item $\int_0^1h_n(r)2\pi rdr=1$.
\end{enumerate}

We consider the Hamiltonian functions $H_n:\D\ra\R$ by $H_n(z)=h_n(|z|)$. Each function $H_n$ defines a time independent vector field $X_n$, whose induced flow is denoted $\phi_n^t$. We have the following property \cite{GHYS} about the computation of the Calabi invariant for compactly supported and autonomous Hamiltonian function  
\begin{prop}
Let us consider $H:\D\ra\R$ a Hamiltonian function with compact support. We denote $\phi^t$ the induced Hamiltonian flow and we have 
$$\Cal(\phi^t)=-2\pi t\int_\D H(z)\omega(z),$$
where $\Cal$ is the Calabi invariant defined by equation \ref{Calintro}.
\end{prop}
This result allows us to compute the Calabi invariant for $\phi^1_n$ and we obtain for each $n\geq 1$
$$\Cal(\phi^1_n)=-2\pi.$$
For each $n\geq 1$ we consider $(\phi^1_n,\id)\in\ti{\Diff}^1_\omega(\D)$ and we have
$$\ti{\Cal}_2((\phi^1_n,\id))=-2\pi.$$
Moreover, $\phi^1_n$ converges to the identity in the $C^0$ topology and we obtain the result.
\end{proof}

			\section{Computation of $\Cal_1$ in some rigidity cases}\label{computation}
	
In this section, we prove several results about the Calabi invariant of irrational pseudo-rotations.\\

		\subsection{A simple case of $C^1$ rigidity}	
			
Let us begin by the simple computation of the Calabi invariant for periodic symplectic maps.
\begin{lem}\label{periodic}
If $f\in \Diff_\omega^1(\D)$ has a finite order, then we have
$$\Cal_1(f)=0.$$
\end{lem}

\begin{proof}
By assumption there exists $p\geq 1$ such that $f^p=\id$ and so $\Cal_1(f^p)=p \ \Cal_1(\id)=0$.
\end{proof}

We deduce the following properties	
\begin{prop}
Let us consider  $f\in \Diff_\omega^1(\D)$. If there exists a sequence of periodic diffeomorphisms $(g_k)_{k\in\N}$ in $\Diff_\omega^1(\D)$ which converges to $f$ for the $C^1$ topology, then we have
$$\Cal_1(f)=0.$$
\end{prop}

\begin{proof}
By Lemma \ref{periodic} for each $n\in \N$ we have $\Cal_1(g_n)=0$ and we obtain the result by the continuity of the map $\Cal_1$ for the $C^1$ topology.
\end{proof}

\begin{prop}\label{C1}
Let us consider  $f\in \Diff_\omega^1(\D)$. If there exists a sequence $(q_k)_{k\in\N}$ such that $f^{q_k}$ converges to the identity in the $C^1$ topology then we have
$$\Cal_1(f)=0.$$
\end{prop}

\begin{proof}
We have $\Cal_1(f^{q_k})=q_k\Cal_1(f)$ and $\Cal_1(f^{q_k})$ converges to $\Cal_1(\id)=0$ so $\Cal_1(f)=0$.
\end{proof}

		\subsection{$C^0$-rigidity}

The following theorem is a stronger version of Corollary \ref{C1}. 	
\begin{theorem}\label{C0cal}
Let us consider  $f\in \Diff_\omega^1(\D)$. If there exists a sequence $(q_k)_{k\in\N}$ of integers such that $(f^{q_k})_{k\in\N}$ converges to the identity for the $C^0$ topology then we have
$$\Cal_1(f)=0.$$
\end{theorem}

To prove the previous statement we will give an estimation of the angle function of $f^{q_n}$ for a given isotopy $I$ from $\id$ to $f$. For that we will consider two cases, the first one if $x$ is \textit{close} to $y$ and the other if $x$ is not \textit{close} to $y$. The following lemma gives us an evaluation of what \textit{close} means.

\begin{lem}\label{closed}
Let us consider $f$ a $C^1$ diffeomorphism of the unit disc $\D$, $I$ an isotopy from $\id$ to $f$. If $d_0(f,\id)\leq\epsilon\leq 1/4$ then for every couple $(x,y)\in\D\times\D$ which satisfies $|x-y|\geq \sqrt{\epsilon},$ we have 
$$|\cos(2\pi\Ang_I(x,y))-1|\leq 4\sqrt{\epsilon}.$$
\end{lem}

\begin{proof}
Let $(x,y)\in\D\times\D$ be a couple such that $|x-y|\geq \sqrt{\epsilon}$. Once can write $f=\id+h$ where $h:\D\ra\R^2$ satisfies $||h||_\infty\leq \epsilon$ and we have

\begin{equation}\label{epsi}
\left|\frac{h(y)-h(x)}{y-x}\right|\leq2\frac{\epsilon}{\sqrt\epsilon}=2\sqrt{\epsilon}.
\end{equation}
We use the equation
\begin{equation}
\cos(2\pi\Ang_I (x,y))=\frac{\langle f(y)-f(x),y-x\rangle}{|f(y)-f(x)|\ |y-x|}
\end{equation}
Moreover, if we write $1=\langle\frac{y-x}{|y-x|},\frac{y-x}{|y-x|}\rangle$ we obtain 
\begin{equation}\label{cosAng}
\cos(\Ang_I (x,y))-1=\langle \frac{f(y)-f(x)}{|f(y)-f(x)|}-\frac{y-x}{|y-x|},\frac{y-x}{|y-x|}\rangle
\end{equation}
Equation \ref{cosAng} becomes 

\begin{align*}
\left| \frac{f(y)-f(x)}{|f(y)-f(x)|}-\frac{y-x}{|y-x|} \right|&\leq |f(y)-f(x)|\left| \frac{1}{|f(y)-f(x)|}- \frac{1}{|y-x|}\right|  +\left| \frac{f(y)-f(x)-(y-x)}{|y-x|}\right|\\
&=\left|\frac{|y-x|-|f(y)-f(x)|}{|y-x|}\right|+\left| \frac{h(y)-h(x)}{|y-x|}\right|\\
&\leq 2\left| \frac{h(y)-h(x)}{y-x}\right|\\
&\leq4\sqrt\epsilon.
\end{align*}
\end{proof}

We obtain the following lemma.
\begin{lem}\label{closeAng}
Under the same hypothesis, there exists an integer $k\in\Z$, uniquely defined, such that for every couple $(x,y)\in\D\times\D$ such that $|x-y|\geq \sqrt{\epsilon}$, we have
\begin{equation}
|\Ang_I(x,y)- k|\leq 2\sqrt[4\,]{\epsilon}/\pi<1/2.
\end{equation}
\end{lem}

\begin{proof}
We consider $\epsilon\in(0,1/16)$ and a couple $(x,y)\in\D$ such that $|y-x|\geq \sqrt\epsilon$. By definition of the floor function there exists a unique $k\in\Z$ such that $2\pi \Ang_I(x,y)-2\pi k \in [-\pi,\pi)$ and we have
$$1\geq \cos(|2\pi \Ang_I(x,y)-2\pi k|)\geq 1-4\sqrt\epsilon\geq 0.$$
The function $\arccos$ is decreasing so we obtain
$$0\leq \arccos(\cos(|2\pi \Ang_I(x,y)-2\pi k|))\leq \arccos(1-4\sqrt\epsilon).$$
The function arccos is defined on $[0,1]$ and of class $C^1$ on $[0,1)$. Moreover we have for every $x\in[0,1)$
$$(\arccos(1-x))'=\frac{1}{\sqrt{2x-x^2}}\leq \frac{1}{\sqrt x}.$$
We obtain that for every $x\in[0,1]$
$$\arccos(1-x)\leq 2\sqrt x.$$
Hence we have
\begin{align*} 
|2\pi \Ang_I(x,y)-2\pi k| &\leq \arccos(1-4\sqrt\epsilon) \\
&\leq 4\sqrt[4\,]{\epsilon}.
\end{align*}
Thus we have
$$|\Ang_I(x,y)-k|\leq 2\sqrt[4\,]{\epsilon}/\pi<1/2.$$
Now we prove that $k$ does not depend of $(x,y)$. Indeed the set of couples $(x,y)\in\D^2$ such that $|x-y|\geq \sqrt\epsilon$ is connected in $\D^2$. Indeed for a couple $(x,y)\in\D^2$ such that $|x-y|\geq \sqrt\epsilon$, let us construct a path from $(x,y)$ to $((-1,0),(1,0))$. \\
We set $d$ the line of $\D^2$ passing through $x$ and $y$. The line $d$ intersects $\Ss^1$ in two points which we denote $\hat x$ and $\hat y$ such that $\hat x$ is closer to $x$ than $y$ and $\hat y$ is closer to $y$ than $x$ as in the previous figure \ref{connect}.\\

 Let us consider the path $\gamma_x:[0,1]\ra \D$ defined by $\gamma_x(t)=t(\hat x-x)+x$ from $x$ to $\hat x$ and the path $\gamma_y:[0,1]$ defines by $\gamma_y(t)=t(\hat y-y)+y$ from $y$ to $\hat y$. So the path $\Gamma:t\mapsto(\gamma_x(t),\gamma_y(t))$ defined on $[0,1]$ sends the couple $(x,y)$ to $(\hat x,\hat y)$.\\
 
Now we  consider $R_\alpha$ the rotation of $\D$ of angle $\alpha=\arg(\hat x)$. Notice that the rotation $R_\alpha^{-1}$ sends $\hat x$  to $(1,0)$. We denote $(R_{t})_{\tin}$ the isotopy from $\id$ to $R_\alpha$ such that for every $\tin$ $R_t$ is the rotation of angle $t\alpha$. Notice that $\hat x=-\hat y$ and so $\hat y$ is send to $(-1,0)$ by $R_\alpha^{-1}$. \\

Hence the composition of the path $\Gamma$ and the path $t\mapsto (R_t^{-1}(\hat x),R_t^{-1}(\hat y))$ sends $(x,y)$ to $((1,0),(-1,0)).$\\

Moreover, $2\sqrt[4\,]{\epsilon}/\pi < 1/2$ so $k$ does not depend on the choice of $(x,y)\in\D$ such that $|x-y|>\sqrt\epsilon$.
 
\end{proof}

With these two lemmas we can give a proof of Theorem \ref{C0cal}.

\begin{proof}[Proof of Theorem \ref{C0cal}]
We can consider $I=(f_t)_{\tin}$ an isotopy from $\id$ to $f$ which fixes a point of $\mathring{\D}$. Up to conjugacy we can suppose that $I$ fixes the origin and we denote $I|_{\Ss^1}$ the restriction of $I$ on $\Ss^1$. We lift $I|_{\Ss^1}$ to an isotopy $(\ti\phi_t)_{\tin}$ on the universal covering space $\R$ of $\Ss^1$ such that $\ti\phi_0=\id$ and set $\ti\phi=\ti\phi_1$. We will prove that $\ti\Cal_2(f,\tphi)=\ti\rho(\tphi)$ and from Theorem \ref{equality} we will obtain
$$\ti\Cal_2(f,\tphi)-\ti\rho(\tphi)=\Cal_1(f) =0.$$

For $q\in\N$ we define the isotopy $I^q$ from $\id$ to $f^q$ as follows. We write $I^q=(f^q_t)_{\tin}$ and for every $z\in\D$ and $t\in[\frac{k-1}{q},\frac{k}{q}]$ we set
$$f^q_t(z)=f_{qt-k+1}\circ \underbrace{(f \circ ...\circ f)}_{k-1 \text{ times}}.$$

 We will denote $\epsilon_n=d_0(f^{q_n},\id)$. For every $k\in\Z$ we can separate the difference between the integral of the angle function of $f^{q_n}$ and $k$ into two parts as follows

\begin{equation}\label{caldecoupe}
\begin{matrix}
\int\int_{\D \times\D} \Ang_{I^{q_{n}}}(x,y)\omega(y)\omega(x) -k =\int_\D \left(\int_{B_{\sqrt{\epsilon_n}}(x)} \Ang_{I^{q_n}} (x,y)\omega(y)-k\right)\omega(x)\\
 \qquad \qquad  \qquad  \qquad  \qquad  \qquad \qquad \qquad  +\int_\D \left(\int_{B^c_{\sqrt{\epsilon_n}}(x)} \Ang_{I^{q_n}} (x,y) \omega(y)-k\right)\omega(x),
\end{matrix}
\end{equation}

where $B^c_{\sqrt{\epsilon_n}}(x)$ is the complementary of $B_{\sqrt{\epsilon_n}}(x)$ in $\D$.\\

We can suppose that $\epsilon_n<1/16$ and by Lemma \ref{closeAng}, there exists a unique $k_n\in\Z$ such that for each couple $(x,y)\in\D\times\D$ such that $|y-x|\geq \sqrt\epsilon_n$ we have
\begin{equation}\label{proche}
|\Ang_{I^{q_n}} (x,y)- k_n|\leq 2\sqrt[4\,]{\epsilon_n}/\pi.
\end{equation}
Moreover, by definition there exists a sequence $(\xi_n)_{n\in\N}$ of $1$-periodic functions $\xi_n:\R\ra\R$ such that $||\xi_1||_\infty\leq 1$ for every $n\in\N$ and such that for every $y\in\Ss^1$ and every lift $\ti y\in\R$ of $y$ we have
$$\Ang_{I^{q_n}}(0,y)=\tphi^{q_n}(\ti{y})-\ti{y}=q_n\ti\rho(\tphi)+\xi_n(\ti{y}).$$
So, for every $y\in\Ss^1$ we have
$$|\Ang_{I^{q_n}}(0,y)-k_n|=|q_n\ti\rho(\tphi) +\xi_n(\ti{y})- k_n|\leq 2\sqrt[4\,]{\epsilon_n}/\pi,$$
where $\ti{y}$ is a lift of $y$. 
Hence we obtain
$$|q_n(\ti\rho(\tphi)-k_n)|\leq 2\sqrt[4\,]{\epsilon_n}/\pi+1.$$
Thus we have 
$$\ti\rho(\tphi)=\lim_{n\ra\infty} \frac{k_n}{q_n}.$$

By equation \ref{proche} we obtain
\begin{equation}\label{int2}
\left|\int_\D\left( \int_{B^c_{\sqrt{\epsilon_n}}(x)} (\Ang_{I^{q_n}} (x,y)-k_n) \omega(y)\right)\omega(x)\right| \leq 4\sqrt[4\,]{\epsilon_n}/\pi.
\end{equation}

We know that for every couple $(x,y)\in\D^2\backslash\Delta$ and for every $n\in\N$ we have
\begin{equation}\label{angsomme}
\Ang_{I^{q_n}}(x,y)=\Ang_{I}(x,y)+\Ang_{I}(f(x),f(y))+...+\Ang_{I}(f^{q_n-1}(x),f^{q_n-1}(y)).
\end{equation}
Hence for every $n\in\N$ the angle function satisfies 
\begin{equation}\label{angmaj}
||\Ang_{I^{q_n}}||_{\infty}\leq q_n ||\Ang_{I}||_{\infty}.
\end{equation}
We can estimate the first integral of equation \ref{caldecoupe} as follows 
\begin{equation}
\left|\int_\D \left( \int_{B_{\sqrt{\epsilon_n}}(x)} (\Ang_{I^{q_n}} (x,y)-k_n)\omega(y)\right)\omega(x)\right|\leq \epsilon_n(q_n||\Ang_I||_\infty+|k_n|).
\end{equation}
So we can deduce from the previous equations a new estimation of the Calabi invariant 

\begin{equation}\label{caldecoupe_n}
\left|\int\int_{\D \times\D} \Ang_{I^{q_{n}}}(x,y)\omega(y)\omega(x)-k_n\right|\leq 2\sqrt[4\,]{\epsilon_n}/\pi +\epsilon_n (q_n||\Ang_f||_\infty +k_n).
\end{equation}
By definition we obtain

\begin{equation}
\left|\ti{\Cal}_2(f,\tphi)-\frac{k_n}{q_n}\right|\leq \frac{2\sqrt[4\,]{\epsilon_n} }{q_n\pi}+ \epsilon_n ||\Ang_f||_\infty+\epsilon_n\frac{k_n}{q_n}.
\end{equation}
Hence we have

\begin{align*}
\left|\ti{\Cal}_2(f,\tphi)-\ti\rho(\ti\phi)\right|&\leq\left|\ti{\Cal}(f,\tphi)-\frac{k_n}{q_n}\right|+\left|\ti\rho(\ti\phi)-\frac{k_n}{q_n}\right|\\
&\leq \frac{4\sqrt[4\,]{\epsilon_n} }{q_n\pi}+ \epsilon_n ||\Ang_f||_\infty+\frac{1}{q_n}+\epsilon_n\frac{k_n}{q_n}.
\end{align*}

By taking the limit on $n\in\N$, we conclude that 
$$\ti{\Cal}_2(f,\tphi)=\ti\rho(\tphi).$$
\end{proof}

\begin{rem}
If we consider a sequence $(\ti g_n=(g_n,\tphi_n))_{n\in\N}\in\ti\Diff^1_\omega(\D)$ which converges to $\ti f =(f,\tphi)\in \ti\Diff^1_\omega(\D)$ in the $C^0$ topology where for each $n\in\N$, $g_n$ is a periodic diffeomorphism of the disk and $f$ is an irrational pseudo-rotation, then the previous method fails to prove that $\ti\Cal_2(g_n,\tphi_n)$ converges to $\ti{\Cal}_2(f,\tphi)$.
It is easy to see that $\Ang_{\ti f}$ is close to $\Ang_{\ti g_n}$ but if we compute the difference $\ti{\Cal}(g_n,\tphi_n)-\ti{\Cal}(f,\tphi),$ as we did in equation \ref{caldecoupe}, we do not have a control of $||\Ang_{\ti g_n}||_\infty$ so we cannot estimate properly the integral 
$$\int_{x}\int_{y\in B_{\sqrt{\epsilon_n}(x)}}\Ang_{\ti g_n}(x,y)\omega(x)\omega(y),$$
where $\epsilon_n=||g_n-f||_\infty$.
\end{rem}

			\section{Examples}\label{examples}

In this section, we will be interested in irrational pseudo rotations with specific rotation numbers. \\	

\textbf{Best approximation:} Let Any irrational number $\alpha\in\R\backslash \Q$ can be written as a continued fraction where $(a_i)_{i\geq1}$ is a sequence of integers $\geq 1$ and $a_0=\lfloor \alpha\rfloor$. Conversely, any sequence $(a_i)_{i\in\N}$  corresponds to a unique number $\alpha$. We define two sequences $(p_n)_{n\in\N}$ and $(q_n)_{n\in\N}$ as follows

\begin{align*}
&p_n=a_np_{n-1}+p_{n_2} \text{ for } n\geq2, & p_0=a_0, \ &p_1=a_0a_1+1\\
&q_n=a_nq_{n-1}+q_{n-2} \text{ for } n\geq2, & q_0=1, \ &q_1=a_1.
\end{align*}

The sequence $(p_n/q_n)_{n\in\N}$ is called the \emph{best approximation} of $\alpha$ and for every $n\geq1$ we have
$$ \{q_{n-1}\alpha\}\leq \{k\alpha\}, \ \forall k<q_n$$
where $\{x\}$ is the fractional part of $x\in\R$. And for every $n\in\N$ we have\\
\begin{equation}\label{Liouville}
\frac{1}{q_n(q_n+q_{n+1})}\leq (-1)^n(\alpha-p_n/q_n,)\leq \frac{1}{q_nq_{n+1}}.
\end{equation}
The numbers $q_n$ are called the \emph{approximation denominators} of $\alpha$.\\

		\subsection{An example of $C^0$ rigidity, the super Liouville type}

In this section, we show that a $C^1$ irrational pseudo rotation with a super Liouville rotation number satisfies the assumptions of Theorem \ref{C0cal}.\\
		
\textit{Super Liouville.} A real number $\alpha\in\R/\Z\backslash\Q$ is called \emph{super Liouville} if the sequence $(q_n)_{n\in\N}$ of the approximation denominators of $\alpha$ satisfies 
\begin{equation}\label{Liouvilledef}
\limsup_n q_n^{-1}\log(q_{n+1})=+\infty.
\end{equation}

If we consider a real $\alpha\in\R$ which has super Liouville type then for each $k\in\Z$ the real $\alpha+k$ is also super Liouville and to simplify the notations we will say that an element $\ti\alpha\in\Tt^1$ is super Liouville.\\

Bramham already showed in \cite{Bram2} that any $C^\infty$ irrational pseudo-rotation $f$ of the disk with super Liouville rotation number is $C^0$ rigid, meaning that $f$ is the $C^0$-limit of a sequence of periodic diffeomorphisms. More recently Le Calvez \cite{CAL1} proved that any $C^1$ irrational pseudo-rotation which is $C^1$ conjugated to a rotation on the boundary is $C^0$ rigid. These results go as follows.

\begin{theorem}\label{decompo}
Let us consider either a $C^\infty$ irrational pseudo rotation or a $C^1$ irrational pseudo rotation $f$ which is $C^1$ conjugated to a rotation on the boundary.We consider $\alpha\in\R$ such that $\alpha+\Z$ is equal to the rotation number of $f$. For a sequence of rationals $(\frac{p_n}{q_n})_{n\in\N}$ which converges to $\alpha$ there exists a sequence $(g_n)_{n\in\N} :\D\ra \D$ of $q_n$-periodic diffeomorphims of the unit disk which converges to $f$ for the $C^0$ topology.\\
 Moreover there exists a constant $C$ depending on $f$ such that for every $n\in\N$ we have
$$d_0(f,g_n)< C(q_n\alpha-p_n)^{\frac{1}{2}}.$$
\end{theorem}

We deduce the following corollary. 

\begin{coro}\label{Liouvillecal}
Let us consider either a $C^\infty$ irrational pseudo rotation or a $C^1$ irrational pseudo rotation $f$ which is $C^1$ conjugated to a rotation on the boundary. If the rotation number of $f$ is super Liouville then we have 
$$\Cal_1(f)=0.$$
\end{coro}

\begin{proof}[Proof of Corollary \ref{Liouvillecal}]
Let us consider $f$ which is either a $C^\infty$ irrational pseudo rotation or a $C^1$ irrational pseudo rotation. We consider $\alpha\in\R$ such that $\alpha+\Z$ is equal to the rotation number of $f$. We will prove that $f$ satisfies the hypothesis of Theorem \ref{C0cal}. We consider $\alpha\in\R$ such that $\alpha+\Z$ is equal to the rotation number of $f$ and we consider a sequence of rationals $(p_n/q_n)_{n\in\N}$ which converges to $\alpha$ such that $q_n$ satisfies equation \ref{Liouville}. Let $(g_n)_{n\in\N}$ be the sequence of $q_n$ periodic diffeomorphisms given by Theorem \ref{decompo} associated to $f$ and the sequence $(p_n/q_n)_{n\in\N}$. We denote by $K$ the $C^1$ norm of $f$ and we set $\epsilon_n=C(q_n\alpha-p_n)^{1/2}$ where $C$ is the constant given by Theorem \ref{decompo}.\\

For all $k\in\N$ and each $n\in\N$ the following inequality holds 
\begin{equation}
d_0(f^k,g_n^k)<K^k\epsilon_n.
\end{equation}
By equation \ref{Liouville} we can majorate $\epsilon_n$ by $\frac{C}{q_{n+1}}$ to obtain for $k=q_n$ the inequality
\begin{equation}\
d(f^{q_n},id) < K^{q_n}\frac{C}{(q_{n+1})^{\frac{1}{2}}}.
\end{equation}
Since $K\geq 1$, equation \ref{Liouvilledef} assures that 
$$\limsup_n \frac{K^{q_n}}{(q_{n+1})^{1/2}}=0.$$
Thus we obtain that
$$\limsup_n d_0(f^{q_n},\id) =0.$$
Hence up to a subsequence we can suppose that 
$$d_0(f^{q_n},id)\ra 0.$$
So $f$ satisfies the hypothesis of Theorem \ref{C0cal} and we conclude
$$\Cal_1(f)=0.$$
\end{proof}

		\subsection{An example of $C^1$-rigidity, the non Bruno type}

\textbf{Bruno type.} A number $\alpha\in\R\backslash\Q$ will be said to be Bruno type if the sequence $(q_n)_{n\in\N}$ of the  approximation denominators of $\alpha$ satisfies
$$\sum_{n=0}^\infty \frac{\log(q_{n+1})}{q_n}<+\infty.$$

If we consider $\alpha\in\R$ which is not Bruno type then for each $k\in\Z$ the real $\alpha+k$ is also not Bruno type and to simplify the notations we will say that an element $\ti\alpha\in\Tt^1$ is non Bruno type.\\
	
A. Avila, B. Fayad, P. Le Calvez, D. Xu and Z. Zhang proved in \cite{AVI} that if we consider a number $\alpha\in\R\backslash\Q$ which is not Bruno type, for $H>1$ there exists a subsequence $q_{n_k}$ of the sequence of the approximation denominators of $\alpha$ such that for every $n\in\N$ $q_{n_{j+1}}\geq H^{q_{n_j}}$ and there exists an infinite set $\mathrm{J}\subset \N$ such that for every $j\in \mathrm{J}$ we have
\begin{equation}\label{Brunotype}
 \{q_{n_j}\alpha\}<\e^{-\frac{q_{n_j}}{j^2}}.
\end{equation}
We can also find the following result in the same paper.

\begin{prop}\label{propnonBruno} 
Let us consider a $C^2$ irrational pseudo rotation $f\in \Diff_\omega^1(\D)$. Suppose that $\rho(f|_{\Ss^1})$ is not Bruno type, then the sequence $q_{n_j}$ satisfies
$$d_1(f^{q_{n_j}},Id)\ra 0.$$

\end{prop} 

Hence a $C^2$ irrational pseudo rotation $f\in \Diff_\omega^1(\D)$ satisfies the hypothesis of Corollary \ref{C1} and we obtain the following corollary.

\begin{coro}\label{nonBruno}
Let us consider  a $C^2$ irrational pseudo rotation $f\in \Diff_\omega^1(\D)$. Suppose that $\rho(f)$ is not Bruno type, then we have 
$$\Cal_1(f)=0.$$
\end{coro}	

\bibliographystyle{plain}
\bibliography{bibliographie}
\end{document}